%% file: reflexive.tex
\documentclass{amsart}
\usepackage{amsmath,amssymb,amsthm,geometry,graphics,color,mdwlist,mathrsfs}
\usepackage[all]{xy}

\input{definitions}

\title{Reflexive modules and Auslander-type conditions}
\author{Norihiro Hanihara}

\thanks{This work is supported by JSPS KAKENHI Grant Numbers JP22KJ0737. The author is supported by Kavli Institute for the Physics and Mathematics of the Universe (WPI) as an affiliate member.}
\address{Faculty of Mathematics, Kyushu University, 744 Motooka, Nishi-ku, Fukuoka, 819-0395, Japan}
\email{hanihara@math.kyushu-u.ac.jp}

\date{}

\begin{document}
\begin{abstract}
We study the category $\refl\La$ of reflexive modules over a two-sided Noetherian ring $\La$. We show that the category $\refl\La$ is quasi-abelian if and only if $\La$ satisfies certain Auslander-type condition on the minimal injective resolution of the ring itself.
Furthermore, we establish a Morita theorem which characterizes the category of reflexive modules among quasi-abelian categories in terms of generator-cogenerators.
%Our result extends the classical Auslander correspondence which characterizes Artin algebras of global dimension at most $2$ such that $\refl\La$ is abelian.
\end{abstract}

\maketitle
\setcounter{tocdepth}{1}
\tableofcontents
\section{Introduction}
Let $\La$ be a (not necessarily commutative) left and right Noetherian ring. A finitely generated module $M$ over $\La$ is called {\it reflexive} if the evaluation map
\[ \xymatrix{ M\ar[r]&\Hom_{\La^\op}(\Hom_\La(M,\La),\La) } \]
is an isomorphism. In what follows we will write $(-)^\ast$ for $\Hom_\La(-,\La)$ or $\Hom_{\La^\op}(-,\La)$, and denote by $\refl\La$
%\[ \refl\La:=\{ \text{reflexive }\La\text{-modules} \}\subset\mod\La, \]
the category of reflexive $\La$-modules. % which is a full subcategory of $\mod\La$, the category of finitely generated module.

Reflexive modules are classically studied in commutative algebra and are also one of the important concepts since the early stage of representation theory \cite{ABr} including the Auslander correspondence \cite{Au71,Iy07b}. They also attract broad modern interest, for example, in the study of non-commutative resolutions, representation theory, and algebraic geometry \cite{VdB04,IR,IW14,BB,Kr24}.

While the definition of reflexivity is quite elementary and the category $\refl\La$ makes sense for arbitrary two-sided Noetherian rings, one needs some constraints on the ring $\La$ for this category $\refl\La$ to behave well. One well-established sufficient condition is that $\La$ should be a commutative normal domain (which is a typical setup in the theory of non-commutative resolutions). It ensures, for example, that $\refl\La$ is precisely the category of second syzygies, and that the inclusion $\refl\La\subset\mod\La$ has a left adjoint (given by the double dual $M\mapsto M^\aast$).

The aim of this paper is to demonstrate that such nice behaviors of the category $\refl\La$ is governed by the {\it Auslander-type condition}, more precisely, the {\it $(2,2)$-condition} \cite{Iy05c} which is a requirement on the minimal injective resolution on the ring $\La$ viewed as a module over itself. Examples of Noetherian rings satisfying the two-sided $(2,2)$-condition include commutative Noetherian normal domains, and Auslander algebras of representation-finite algebras.

Recall that for each pair of positive integers $l$ and $n$, we say that $\La$ satisfies the {\it $(l,n)$-condition} \cite{Iy05c} if, in the minimal injective resolution
\[ \xymatrix{ 0\ar[r]&\La\ar[r]&I^0\ar[r]&I^1\ar[r]&\cdots } \]
of $\La$ in $\Mod\La$, we have $\fld I^i<l$ for all $0\leq i<n$. We say that $\La$ satisfies the {\it two-sided $(l,n)$-condition} if both $\La$ and $\La^\op$ satisfy the $(l,n)$-condition.
For example, the $(1,n)$-condition is nothing but the {dominant dimension} \cite{Tac} of $\La$ being at least $n$. Also, $\La$ is {$n$-Gorenstein} \cite{FGR} precisely when $\La$ satisfies the $(l,l)$-condition for all $1\leq l\leq n$.

%Such Auslander-type conditions originate in exploring non-commutative analogues of Gorenstein ring, and are closely related to such concepts as Auslander-Gorenstein/Auslander-regular rings \cite{FGR}, Artin-Schelter Gorenstein/regular algebras appearing in non-commutative algebraic geometry \cite{AS,ATV}, or Calabi-Yau algebras \cite{Gi,Ke11} in recent cluster theory, and also (higher) Auslander algebras \cite{Au71,Iy07b}.
Such Auslander-type conditions give one of the non-commutative analogues of Gorenstein ring, and appear in various contexts including the study of Auslander correspondence, non-commutative algebraic geometry, and cluster theory. We refer to \cite{AS,ATV,Au71,AR94,AR96,FGR,Gi,Iy05,Iy07b,Ke11,VdB04} for some of such studies.
%\comment{Refer to Auslander-Gorestein rings, Auslander-regular, Artin-Schelter regular algabras, CY algebras, as non-commutative Gorenstein/regular algebras}

%Consider the minimal injective resolution
%\[ \xymatrix{ 0\ar[r]&\La\ar[r]&I^0\ar[r]&I^1\ar[r]&\cdots } \]
%of $\La$ in $\Mod\La$. For each positive integer $l$ and $n$, we say that $\La$ satisfies the {\it $(l,n)$-condition} if $\fld I^i<l$ for all $0\leq i<n$. We say that $\La$ satisfies the {\it two-sided $(l,n)$-condition} if both $\La$ and $\La^\op$ satisfies the $(l,n)$-condition.
%

The main result of this paper is the following characterization of when the category $\refl\La$ is quasi-abelian in the sense explained below.
\begin{Thm}[=Theorem \ref{max}]\label{imax}
Let $\La$ be a two-sided Noetherian ring. Then the category $\refl\La$ is quasi-abelian if and only if $\La$ satisfies the two-sided $(2,2)$-condition. %,that is, $\fld(I^0\oplus I^1)\leq 1$.
\end{Thm}
Here, an additive category $\C$ is {\it quasi-abelian} if it has kernels and cokernels, kernels are closed under push-outs, and cokernels are closed under pull-backs. It implies that $\C$ has the maximum exact structure given by the class of all kernel-cokernel pairs. This means that the category $\C$ has an intrinsic homological algebra.

%Moreover, the following results show that when $\refl\La$ is quasi-abelian, we have all the desired property of the category of reflexives modules.
To be explicit, we record the following summary of the desirable behaviors of the category $\refl\La$ in the setting of Theorem \ref{imax}. Note that (1) and (2) appear in \cite[Theorem 1.7]{AR96}.
\begin{Prop}
In the situation of Theorem \ref{imax}, we have the following.
\begin{enumerate}
\item $\refl\La$ coincides with the category of second syzygies in $\mod\La$.
\item $\refl\La\subset\mod\La$ is closed under kernels and extensions.
\item The inclusion $\refl\La\subset\mod\La$ has a left adjoint given by the double dual $X\mapsto X^\aast$.
\end{enumerate}
\end{Prop}

We further discuss when the category $\refl\La$ is moreover abelian. In fact we have the following simple and analogous characterization.
\begin{Thm}[=Theorem \ref{abelian}]\label{iabelian}
Let $\La$ be a Noetherian ring. Then category $\refl\La$ is abelian if and only if it satisfies the $(1,2)$-condition, in other words, $\ddim\La\geq2$.
\end{Thm}
Note that by left-right symmetry of the dominant dimension \cite{Ho}, the $(1,2)$-condition and the two-sided $(1,2)$-condition are equivalent. Furthermore, we know that any Noetherian ring of dominant dimension greater than $1$ is Artinian \cite[Proposition 7]{IwS}, thus what Theorem \ref{iabelian} gives is the Morita-Tachikawa correspondence (see Theorem \ref{MT}) for Artinian rings (not necessarily module-finite over the center).

Apart from the maximum exact structure in Theorem \ref{imax}, we will present in Theorem \ref{subset} how to produce more exact structures on $\refl\La$ from certain Serre subcategories of $\mod\La$.

\medskip

Now, our main result Theorem \ref{imax} motivates us to study {\it Morita theory} for quasi-abelian categories. Most classically, Morita theorem gives a characterization of module categories among abelian categories in terms of projective generators \cite{Ga}. In particular, it yields a complete characterization of rings having equivalent module categories. Such rings are called Morita equivalent, and Morita theory can be understood as describing the structure of the category of rings up to Morita equivalence.
%Beyond the classical case of module categories, Morita theory

Our Morita theorem for quasi-abelian categories gives a characterization of the category of reflexives among quasi-abelian categories. We say that an object $M$ in an exact category is a {\it generator-cogenerator} if each object admits a deflation from (resp. an inflation into) an object in $\add M$.
\begin{Thm}[=Theorem \ref{Morita}]\label{iMorita}
Let $\C$ be a quasi-abelian category. Suppose that there is a generator-cogenerator $M\in\C$ such that $\La:=\End_\C(M)$ is two-sided Noetherian. Then $\La$ satisfies the two-sided $(2,2)$-condition and the functor $\Hom_\C(M,-)\colon\C\to\refl\La$ is an equivalence.
\end{Thm}
Consequently, we obtain a characterization of {\it reflexive equivalence} of Noetherian rings (Theorem \ref{refeq}), that is, those having equivalent categories of reflexives. We refer to \cite{IR} for a result for module-finite algebras.

\subsection*{Acknowledgements}
The author is deeply grateful to Alexey Bondal, Martin Kalck, Henning Krause, Osamu Iyama, Bernhard Keller, Hiroki Matsui, and Shinnosuke Okawa for stimulating discussions and valuable comments.
He would like to thank the anonymous referee for careful reading of the manuscript and numerous suggestion which improved presentations.
%Valuable and stimulating discussions: Alexey Bondal, Henning Krause, Osamu Iyama, Shinnosuke Okawa
%The author is grateful to Alexey Bondal, Martin Kalck, Henning Krause, Osamu Iyama, Bernhard Keller, Hiroki Matsui, Tsutomu Nakamura, Shinnosuke Okawa, Julia Sauter, and Ryo Takahashi for valuable discussions and comments on the draft.

\section{Auslander-type conditions}\label{AG}%{Auslander-type conditions and Auslander-Gorenstein rings}
Throughout this paper, a Noetherian ring will mean a (not necessarily commutative) left and right Noetherian ring unless stated otherwise.

\subsection{Auslander-type conditions and Auslander-Gorenstein rings}
We review the concept of {\it Auslander-type conditions} following \cite{Iy05c}. It is a vast generalization of dominant dimension \cite{Tac} which plays a prominent role in Auslander correspondences \cite{Au71,Iy07b}, and at the same time covers the notion of {\it Auslander-Gorenstein rings} \cite{AR94,FGR}.

Before that we recall the notion of (strong) grade which gives an important reformulation of Auslander-type conditions and will be used throughout the paper.
\begin{Def}
Let $\La$ be a Noetherian ring and $M\in\Mod\La$.
\begin{enumerate}
\item The {\it grade} of $M$, denoted $\grade M$, is the minimum integer $n\geq0$ such that $\Ext^n_\La(M,\La)\neq0$.
\item The {\it strong grade} of $M$, denoted $\sgrade M$, is defined as
\[ \sgrade M=\inf\{n\geq0\mid \Ext^n_\La(N,\La)\neq0 \text{ for some submodule }N\subset M\}. \]
In other words, we have $\sgrade M\geq n$ if and only if $\grade N\geq n$ for all submodules $N\subset M$.
\end{enumerate}
\end{Def}
%Recall that the {\it grade} of a $\La$-module $M$ is the minimum integer $n\geq0$ such that $\Ext^n_\La(M,\La)\neq0$. Recall also the following stronger notion which is important for a characterization of $n$-Gorenstein rings:

The following description of the strong grade is useful.
\begin{Prop}\label{sgrade}
	Let $X\in\Mod\La$ and $0\to\La\to I^0\to I^1\to\cdots$ be the minimal injective resolution. Then
	\[ \sgrade X=\inf\{n\geq0 \mid \Hom_\La(X,I^n)\neq0\}. \]
	In particular, for each $n\geq0$ the subcategory $\{X\in\Mod\La\mid \sgrade X\geq n\}\subset\Mod\La$ is a Serre subcategory.
\end{Prop}
%\begin{proof}
%	\comment{Let us include the proof for the convenience of the reader.}
%\end{proof}

The following computation of the flat dimension of injective modules is also fundamental.
\begin{Prop}[{\cite[2.1]{AR96}}]\label{TorI}
Let $I$ be an injective module over a Noetherian ring $\La$. Then for all $X\in\mod\La^\op$ and $n\geq0$ we have an isomorphism
\[ \Tor_n^\La(X,I)=\Hom_\La(\Ext^n_{\La^\op}(X,\La),I). \]
Therefore, one has $\fld I<l$ if and only if $\Hom_\La(\Ext^l_{\La^\op}(X,\La),I)=0$ for all $X\in\mod\La^\op$.
\end{Prop}
\begin{proof}
	This is easily seen by computing from the projective resolution of $X$ and using the formula $P\otimes_\La I=\Hom_\La(P^\ast,I)$ for $P\in\proj\La^\op$.
\end{proof}

We are now ready to state the definition of the Auslander-type condition. The equivalence of the conditions below follows easily from Proposition \ref{sgrade} and Proposition \ref{TorI}.
\begin{Def}[{\cite{Iy05c}}]
Let $l$ and $n$ be positive integers. The {\it $(l,n)$-condition} on a Noetherian ring $\La$ is the following equivalent conditions, where $0\to\La\to I^0\to I^1\to\cdots$ is the minimal injective resolution in $\Mod\La$.
\begin{enumerate}
	\renewcommand\labelenumi{(\roman{enumi})}
	\renewcommand\theenumi{\roman{enumi}}
	\item $\fld I^i<l$ for all $0\leq i<n$.
	\item $\sgrade\Ext^l_{\La^\op}(X,\La)\geq n$ for all $X\in\mod\La^\op$.
\end{enumerate}
The {\it $(l,n)^\op$-condition} on $\La$ is by definition the $(l,n)$-condition on $\La^\op$. We say that $\La$ satisfies the {\it two-sided $(l,n)$-condition} if it satisfies the $(l,n)$- and $(l,n)^\op$-conditions. 
\end{Def}

The $(l,n)$-conditions include the following notion of $n$-Gorenstein rings.
\begin{Def}[{\cite{AR94,FGR}}]
Let $n\geq0$ be an integer. A Noetherian ring $\La$ is {\it $n$-Gorenstein} if in the minimal injective resolution
\[ \xymatrix{ 0\ar[r]&\La\ar[r]&I^0\ar[r]&I^1\ar[r]&\cdots }, \]
one has $\fld I^i\leq i$ for all $0\leq i<n$. We say that a Noetherian ring $\La$ is {\it Auslander-Gorenstein} if it is $n$-Gorenstein for all $n$ and $\id\La<\infty$ on each side. Similarly, we say $\La$ is {\it Auslander-regular} if it is $n$-Gorenstein for all $n$ and $\gd\La<\infty$.
\end{Def}
Indeed, $\La$ is $n$-Gorenstein if and only if it satisfies the $(l,l)$-condition for all $1\leq i\leq n$.
It is known that being $n$-Gorenstein is left-right symmetric \cite[Theorem 3.7]{FGR}. 

%\comment{
%We end this section with the following important example.
%\begin{Ex}
%Let $\Si$ be a finite dimensional algebra of finite representation type with an additive generator $M$ of $\mod\Si$. Let $\La=\End_\Si(M)$, the Auslander algebra of $\La$. This 
%\end{Ex}
%}

\subsection{Commutative rings}
%We list some well-known result on Auslander-type conditions for commutative rings.
Let us specialize the above notion to commutative rings which are one of our important examples. This will show that Auslander-type conditions indeed generalize commutative Gorenstein rings.
Throughout this subsection we let $R$ be a commutative Noetherian ring. We refer to \cite{RF} for more information. Let us include the proof of the following result, which is one of the most fundamental results concerning the flat dimension of injective modules over commutative rings. %It follows that any commutative Gorenstein ring is Auslander-Gorenstein.
\begin{Lem}\label{fd}
Let $\p\in\Spec R$. Then $\fld E_R(R/\p)$ is finite if and only if $R_\p$ is Gorenstein. In this case, one has $\fld E_R(R/\p)=\height\p$.
\end{Lem}
\begin{proof}
	%We use the following to compute the flat dimension of an injective module $I$: For a Noetherian ring $\La$ and $M\in\mod\La^\op$, there is an isomorphism $\Tor_i^\La(M,I)=\Hom_\La(\Ext^i_{\La^\op}(M,\La),I)$ (cf.\! \cite[2.1]{AR96}).
	
	%Let $\p,\q\in\Spec R$ and consider the $R$-module $\Tor_i^R(E_R(R/\p),R/\q)=\Hom_R(\Ext^i_{R}(R/\q,R),E_R(R/\p))$. The flat dimension of $E_R(R/\p)$ is $\leq n$ if and only if this is $0$ for all $\q\in\Spec R$ and $i>n$. Since $\Ext^i_R(R/\q,R)$ is finitely generated and $\Ass E_R(R/\p)=\{\p\}$, the set of associated primes of the right-hand-side is contained in $\{\p\}$, thus it is $0$ if and only if it is annihilated by $-\otimes_R R_\p$. Localizing at $\p$ yields $\Hom_{R_\p}(\Ext^i_{R_\p}(R_\p/\q R_\p,R_\p),E_{R_\p}(R_\p/\p R_\p))$, which allows us to replace $R$ by the local ring $(R_\p,\p R_\p)$. Then this space is $0$ for $i>n$ if and only if $\Ext^i_{R_\p}(R_\p/\q R_\p,R_\p)=0$ for all $\q\in\Spec R_\p$ and $i>n$, which is exactly the case when $R_\p$ is Gorenstein of dimension $\leq n$.
	We use Proposition \ref{TorI} to compute the flat dimension of an injective module.
	Let $\p\in\Spec R$. The flat dimension of $E_R(R/\p)$ is at most $n$ if and only if $\Tor_i^R(M,E_R(R/\p))=\Hom_R(\Ext^i_{R}(M,R),E_R(R/\p))$ is $0$ for each $M\in\mod R$ and $i>n$. Since $\Ext^i_R(M,R)$ is finitely generated and $\Ass E_R(R/\p)=\{\p\}$, the set of associated primes of the right-hand-side is contained in $\{\p\}$, thus it is $0$ if and only if it is annihilated by $-\otimes_R R_\p$. Localizing this at $\p$ yields $\Hom_{R_\p}(\Ext^i_{R_\p}(M_\p,R_\p),E_{R_\p}(R_\p/\p R_\p))$, which is $0$ if and only if $\Ext^i_{R_\p}(M_\p,R_\p)=0$. Noting that the functor $-\otimes_RR_\p\colon\mod R\to \mod R_\p$ is essentially surjective, we conclude that $\fld E_R(R/\p)\leq n$ exactly when $R_\p$ is Gorenstein of dimension $\leq n$.
\end{proof}

\begin{Lem}\label{fd=n}
Let $0\to R\to I^0\to I^1\to\cdots$ be the minimal injective resolution. Suppose that $\fld I^n<\infty$ for some $n\geq0$. Then every indecomposable summand of $I^n$ has flat dimension precisely $n$.
\end{Lem}
\begin{proof}
	Let $I=E_R(R/\p)$ be any indecomposable summand of $I^n$, and let $m$ be its flat dimension. By Lemma \ref{fd} we have that $R_\p$ is an $m$-dimensional Gorenstein ring. Also, localizing the resolution we have the minimal injective resolution $0\to R_\p\to I^0_\p\to I^1_\p\to\cdots\to I^n_\p\to\cdots$ with $I_\p=E_{R_\p}(R_\p/\p R_\p)$ a direct summand of $I^n_\p$. We then conclude that $n=\dim R_\p=m$.
\end{proof}

Recall the following conditions on a ring $R$:
\begin{description}
\item[$(S_n)$] $\depth R_\p\geq\min\{n,\height\p\}$ for all $\p\in\Spec R$,
%\item[$(T_n)$] $R_\p$ is Gorenstein for all $\p\in\Spec R$ with $\height\p\leq n$.
%\item[$(G_n)$] $R$ is $n$-Gorenstein.
\item[$(G_n)$] $R_\p$ is Gorenstein for all $\p\in\Spec R$ with $\height\p\leq n$.
\end{description}
Note that our notation is different from \cite{RF}.
The following result characterizes commutative $n$-Gorenstein rings.
\begin{Prop}[\cite{Isch,RF}]\label{comm}
For a commutative Noetherian ring $R$ and $n\geq1$, the following are equivalent.
\begin{enumerate}
\renewcommand\labelenumi{(\roman{enumi})}
\renewcommand\theenumi{\roman{enumi}}
\item\label{NG} $R$ is $n$-Gorenstein, that is, satisfies the $(l,l)$-condition for each $1\leq l\leq n$.
\item\label{LN} $R$ satisfies the $(l,n)$-condition for some $l\geq1$.
\item\label{SG} $R$ satisfies the $(G_{n-1})$- and $(S_n)$-conditions.
\end{enumerate}
%A commutative Noetherian ring is $n$-Gorenstein if and only if it satisfies $(S_n)$- and $(G_{n-1})$-conditions.
\end{Prop}
\begin{proof}
	We include the proof for the convenience of the reader. We denote by $0\to R\to I^0\to I^1\to \cdots$ the minimal injective resolution of $R$.
	
	If $R$ is $n$-Gorenstein, then it satisfies in particular the $(n,n)$-condition, thus we have (\ref{NG})$\Rightarrow$(\ref{LN}). Conversely, if $\fld I^i<\infty$ for $0\leq i<n$, then $\fld I^i=i$ by LEmma \ref{fd=n}, so we have (\ref{LN})$\Rightarrow$(\ref{NG}).
	
	We next prove (\ref{NG})$\Rightarrow$(\ref{SG}). We first prove the $(G_{n-1})$-condition. Let $\p\in\Spec R$ be of height $m\leq n-1$. Note that by $\Ext^m_{R_\p}(R_\p/\p R_\p,R_\p)\neq0$, we have $\p\in\Ass I^m$. Then since $\fld I^m$ is finite, we deduce by Lemma \ref{fd} that $R_\p$ is Gorenstein.
	We now turn to the $(S_n)$-condition. By the $(G_{n-1})$-condition, it is enough to show that if $\height\p\geq n$ then $\depth R_\p\geq n$. The $(G_{n-1})$-condition also implies that $I^i=\bigoplus_{\height\p=i}E_R(R/\p)$ for all $0\leq i\leq n-1$. This shows any height $\geq n$ prime is associated to $I^j$ with $j\geq n$, which yields $\depth R_\p\geq n$.
		
	%Note that since the localization functor $\mod R\to\mod R_\p$ is essentially surjective, the $(l,n)$-condition localizes, hence so does the $n$-Gorenstein property. Let $\p\in\Spec R$. We have the minimal injective resolution
	%\[ \xymatrix{ 0\ar[r]&R_\p\ar[r]&I^0_\p\ar[r]&\cdots\ar[r]& I^{n-1}_\p\ar[r]&I^n_\p } \]
	%in $\Mod R_\p$.
	%Since $\fld I^{m}_\p$ is finite, the same holds for every indecomposable summand by Lemma \ref{fd=n}, thus $\Ass I^{m}_\p$ consists of height $m$ primes by Lemma \ref{fd}. But the maximal ideal $\p R_\p$ of $R_\p$ is the unique such ideal, so we deduce that $R_\p$ is Gorenstein. This proves the $(G_{n-1})$-condition.
	
	Finally, we show (\ref{SG})$\Rightarrow$(\ref{NG}). Recall that for $\p\in\Spec R$, the number of summands $E_R(R/\p)$ appearing in $I^i$ is given by $\dim_{k(\p)}\Ext^i_{R_\p}(k(\p),R_\p)$. We see by the $(S_n)$-condition that if $\Ass(I^0\oplus\cdots\oplus I^{n-1})$ consists of prime ideals of height $\leq n-1$. Also, the $(G_{n-1})$-condition shows $I^i=\bigoplus_{\height\p=i}E_R(R/\p)$ for every $0\leq i\leq n-1$. Then Lemma \ref{fd} implies that $R$ is $n$-Gorenstein.
\end{proof}

%We get an interpretation of normality in terms of Auslander-Gorenstein rings.
We end this section with the following typical example. We say that $R$ is {\it generically Gorenstein} if $R_\p$ is Gorenstein for all $\p\in\Ass R$. Also, the {\it $(R_n)$-condition} requires $R_\p$ to be regular for all $\p\in\Spec R$ with $\height\p\leq n$.
\begin{Ex}\label{2-Gor}
Consider a commutative Noetherian ring.
\begin{enumerate}
\item Any domain is $1$-Gorenstein. In fact, we have the following implications:
\[ \text{ domain }\Longrightarrow\text{ generically Gorenstein }\Longleftrightarrow 1\text{-Gorenstein}. \] 
\item Any normal domain is $2$-Gorenstein. In fact, we have the following implications:
\[ \text{ normal } \Longleftrightarrow (R_1)+(S_2)\Longrightarrow (G_1)+(S_2)\Longleftrightarrow 2\text{-Gorenstein}. \]
%Indeed, it is well-known that a commutative Noetherian domain is normal if and only if it satisfies the $(R_1)$- and $(S_2)$-condition.
\end{enumerate}
\end{Ex}
\section{The category of reflexive modules}
\subsection{Reflexive modules over Noetherian rings with the two-sided $(2,2)$-condition}
It is well-known (e.g.\! \cite{IR,IW14}) that reflexive modules over normal domains have a nice behavior. We observe in this section that the theory works equally well in the setting of Noetherian rings satisfying the two-sided $(2,2)$-condition. Let us review some basic notions.
\begin{Def}\label{n}
Let $\La$ be a Noetherian ring, $X\in\mod\La$, and $n\geq1$ an integer.
\begin{enumerate}
\item $X$ is an {\it $n$-th syzygy} if there exists an exact sequence $0\to X\to Q^0\to\cdots \to Q^{n-1}$ with $Q^i\in\proj\La$.
\item $X$ is {\it $n$-torsion free} if one has $\Ext^i_{\La^\op}(\Tr X,\La)=0$ for all $1\leq i\leq n$.
%\item $X$ is {\it reflexive} if the canonical morphism $X\to\Hom_{\La^\op}(\Hom_\La(X,\La),\La)$ is an isomorphism.
\end{enumerate}
We use the terms {\it torsion-free} for $1$-torsion freeness, and {\it reflexive} for $2$-torsion freeness.
\end{Def}
We denote by $(-)^\ast$ the functors $\Hom_\La(-,\La)\colon\mod\La\leftrightarrow\mod\La^\op$.
The Auslander-Bridger sequence \cite{ABr} below justifies the terminologies ``torsion-free'' and ``reflexive''.
\[ \xymatrix{ 0\ar[r]&\Ext^1_{\La^\op}(\Tr X,\La)\ar[r]&X\ar[r]& X^{\ast\ast}\ar[r]&\Ext^2_{\La^\op}(\Tr X,\La)\ar[r]&0 } \]
%shows that $X$ is torsion-free if and only if 

It is easy to see that any $n$-torsion free module is an $n$-th syzygy. The converse is a deep problem which led to the study of Auslander-Gorenstein rings \cite{AR94,AR96}. For each positive integer $n$ we denote by $\Om^n$ the category of $n$-th syzygies.
\begin{Prop}[{\cite[Theorem 0.1, Theorem 1.7]{AR96}}]\label{AR}
%Let $\La$ be a $1$-Gorenstein ring. Then $X\in\mod\La$ is reflexive if and only if it is a second syzygy. If moreover $\La$ is $2$-Gorenstein, then $\refl\La\subset\mod\La$ is closed under extensions.
Let $\La$ be a Noetherian ring.
\begin{enumerate}
\item $\Om^2=\refl\La$ if and only if $\grade\Ext^2_\La(X,\La)\geq1$ for all $X\in\mod\La$.
\item\label{Om^1} $\Om^1\subset\mod\La$ is extension-closed if and only if $\sgrade\Ext^2_\La(X,\La)\geq1$ for all $X\in\mod\La$ if and only if $\grade\Ext^1_{\La^\op}(Y,\La)\geq1$ for all $Y\in\mod\La^\op$.
\end{enumerate}
\end{Prop}

We note an easy observation which we use later.
\begin{Lem}\label{easy}
Let $n\geq0$ and suppose that $\grade E\geq n$. Then $\Ext^i_\La(E,M)=0$ for all $n$-th syzygies $M$ and $0\leq i<n$.%$\Hom_\La(E,M)=0$ and $\Ext^1_\La(E,M)=0$ for all $2$-nd syzygy $M$.
\end{Lem}
\begin{proof}
	Pick an exact sequence $0\to M\to Q_0\to Q_1\to\cdots\to Q_{n-1}$ with $Q_i\in\proj\La$ and put $M_i:=\Im(Q_{i-1}\to Q_i)$. Then applying $\Hom_\La(E,-)$ we have isomorphisms $\Ext^{i}_\La(E,M)\ysimeq\Ext^{i-1}_\La(E,M_1)\ysimeq\cdots\ysimeq\Ext^1_\La(E,M_{i-1})\ysimeq\Hom_\La(E,M_i)$ by $\grade E\geq i$, and $\Hom_\La(E,M_i)\hookrightarrow\Hom_\La(E,Q_i)=0$.
\end{proof}

%\begin{Rem}
%Suppose that $\sgrade\Ext^2_{\La^\op}(X,\La)\geq1$ for all $X\in\mod\La^\op$. Then $\grade\Ext^1_{\La}(M,\La)\geq1$ for all $M\in\mod\La$.
%\end{Rem}
%\begin{proof}
%	Let $0\to N\to P\to M\to 0$ be an exact sequence in $\mod\La$ with $P\in\proj\La$. It gives an exact sequence $0\to M^\ast\to P^\ast\to N^\ast\to \Ext^1_\La(M,\La)\to 0$. We want to show that the map $N^\aast\to P^\aast$ is injective. Let $K$ be the kernel and consider the commutative diagram
%	\[ \xymatrix{
%		&0\ar[r]&N\ar[r]\ar[d]&P\ar[r]\ar@{=}[d]&M\\
%		0\ar[r]&K\ar[r]&N^\aast\ar[r]&P^\aast.} \]
%	Since $K\cap N=\Ker(N\to P)=0$, we see that $K$ maps injectively to $\Coker(N\to N^\aast)=\Ext^2_{\La^\op}(\Tr N,\La)$. Now, $\Ext^2_{\La^\op}(\Tr N,\La)$ has $\sgrade\geq1$ by assumption, so $\grade K\geq1$. Since $K$ is torsion-free, this forces $K=0$, as desired. 
%\end{proof}

\subsection{Exact categories}
Let us start by recalling the notion of Quillen's {\it exact categories}. We call a sequence $L\xrightarrow{f}N\xrightarrow{g}M$ in an additive category a {\it kernel-cokernel pair} if $f=\ker g$ and $g=\coker f$.
\begin{Def}\label{exdef}
An {\it exact structure} on an additive category $\E$ is a class $\S$ of kernel-cokernel pairs in $\E$ satisfying the following conditions, where we call a kernel-cokernel pair $L\xrightarrow{f}N\xrightarrow{g}M$ in $\S$ a {\it conflation}, and we say $f$ is an {\it admissible monomorphism} or an {\it inflation}, and $g$ is an {\it admissible epimorphism} or a {\it deflation}.
\begin{enumerate}
\renewcommand\labelenumi{(\roman{enumi})}
\renewcommand\theenumi{\roman{enumi}}
\item\label{EX1} The class $\S$ is closed under isomorphisms, and the trivial sequences $L\xrightarrow{1}L\to0$ and $0\to L\xrightarrow{1}L$ are in $\S$.
\item\label{EX2} Admissible monomorphisms (resp. admissible epimorphisms) are closed under compositions.
\item\label{EX3} For any conflation $L\xrightarrow{f}N\xrightarrow{g}M$ and any morphism $a\colon L\to L^\prime$ (resp. $c\colon M^\prime\to M$), there exists a push-out of $f$ along $a$ (resp. a pull-back of $g$ along $c$), and the obtained sequences below are conflations.
\[ 
\xymatrix{
	L\ar[r]^-f\ar[d]_-a&N\ar[r]^-g\ar[d]&M\ar@{=}[d]\\
	L^\prime\ar[r]&N^\prime\ar[r]&M}
\qquad
\xymatrix{
	L\ar[r]\ar@{=}[d]&N^\pprime\ar[r]\ar[d]&M^\prime\ar[d]^-c\\
	L\ar[r]^-f&N\ar[r]^-g&M }
\]
\end{enumerate}
An {\it exact category} is an additive category endowed with an exact structure.
\end{Def}

Let us also recall the following notion which plays an important role later.
\begin{Def}
A {\it quasi-abelian category} is an additive category satisfying the following.
\begin{enumerate}
\renewcommand\labelenumi{(\roman{enumi})}
\renewcommand\theenumi{\roman{enumi}}
\item Any morphism has a kernel and a cokernel.
\item Kernels are closed under push-outs and cokernels are closed under pull-backs.
\end{enumerate}
\end{Def}

Notice that the set of exact structures on an additive category forms a poset with respect to the inclusion. The {\it maximum exact structure} on an additive category is the maximum (=unique maximal) elmement of this poset. By \cite{Ru}, any additive category has the maximum exact structure.
The following result describes the maximum exact structure for quasi-abelian categories.
%It is related to exact categories as follows.
\begin{Prop}[{See \cite[Proposition 4.4]{Bu}}]
Let $\E$ be a quasi-abelian category. Then the class of all kernel-cokernel pairs forms an exact structure on $\E$.
\end{Prop}

%We note that the same conclusion does not hold for $n$-torsion free modules with $n\geq3$, that is, in the category of $n$-torions free modules over an $n$-Gorenstein ring, the class of all short exact sequence does not necessarily form an exact structure.
%\begin{Ex}
%Let $R$ be a commutative reguar ring of dimension $n\geq3$, so that $n$-torsion free modules are just projective modules. According to \cite{E}, the short exact sequence in $\proj R$ corresponds to Cohen-Macaulay $R$-modules of codimension $2$, which does not form a Serre subcategory of $\mod R$.
%\end{Ex}

\subsection{The maximum exact structure}
The main result of this section is the following characterization of when $\refl\La$ is quasi-abelian in terms of Auslander-type conditions.
\begin{Thm}\label{max}
Let $\La$ be a Noetherian ring.
The category $\refl\La$ is quasi-abelian if and only if $\La$ satisfies the two-sided $(2,2)$-condition.
\end{Thm}

Let us discuss the existence of kernels and cokernels in the category $\refl\La$. Note that, since we have a duality $(-)^\ast\colon\refl\La\leftrightarrow\refl\La^\op$ for arbitrary Noetherian rings, the conditions in Lemma \ref{EX-ker} and in Lemma \ref{EX-coker} are opposite to each other, that is, the conditions for $\La$ in Lemma \ref{EX-ker} are equivalent to the ones for $\La^\op$ in Lemma \ref{EX-coker}.
\begin{Lem}\label{EX-ker}
Let $\La$ be a Noetherian ring. The following are equivalent.
\begin{enumerate}
\romanenum
\item\label{ex-ker} $\refl\La$ has kernels.
\item\label{cl-ker} $\refl\La\subset\mod\La$ is closed under kernels.
\item\label{ref=Om^2} $\refl\La=\Om^2$.
\item\label{grade1} $\grade\Ext^2_\La(X,\La)\geq1$ for all $X\in\mod\La$.
\end{enumerate}
In this case, the kernel in $\refl\La$ is computed as the kernel in $\mod\La$.
\end{Lem}
\begin{proof}
	The assertion (\ref{cl-ker})$\Rightarrow$(\ref{ex-ker}) is clear. To see (\ref{ex-ker})$\Rightarrow$(\ref{cl-ker}), let $L\to M$ be a morphism in $\refl\La$ and let $K\to L$ be its kernel in $\refl\La$, thus we have a sequence $K\xrightarrow{} L\xrightarrow{} M$. Since $\La\in\refl\La$, applying $\Hom_{\refl\La}(\La,-)$ yields an exact sequence $0\to K\to L\to M$, which shows $\refl\La\subset\mod\La$ is closed under kernels.
	
	To see (\ref{cl-ker})$\Rightarrow$(\ref{ref=Om^2}), we only have to prove $\Om^2\subset\refl\La$ since the converse inclusion always holds (that is, each $2$-torsion free module is a second syzygy). Note that each $M\in\Om^2$ is a kernel of a map between projective modules. Since projective modules are reflexive, the assumption (\ref{cl-ker}) implies $M\in\refl\La$.
		
	We next prove (\ref{ref=Om^2})$\Rightarrow$(\ref{cl-ker}) . Let $f\colon L\to M$ be a map between reflexive modules. Since $L$ is $2$-torsion free, there is an exact sequence $0\to L\to P^0\to P^1$ which stays exact after applying $(-)^\ast$, in other words, the maps $a\colon L\to P^0$ and $b\colon \Coker a\to P^1$ are left $(\proj\La)$-approximations. Taking the sequence $0\to M\to Q^0\to Q^1$ for $M\in\Om^2$, we may lift $f\colon L\to M$ to a map of exact sequences as in the diagram below since the maps $a$ and $b$ are apporoximations.
	\[ \xymatrix{
		0\ar[r]&L\ar[r]\ar[d]_-f&P^0\ar[r]\ar@{-->}[d]&P^1\ar@{-->}[d]\\
		0\ar[r]&M\ar[r]&Q^0\ar[r]&Q^1 } \]
	Then we get an exact sequence $0\to \Ker f\to P^0\to P^1\oplus Q^0$, which shows that $\Ker f\in\Om^2$, thus $\Ker f\in\refl\La$ by assumption (\ref{ref=Om^2}).
		 
	%The assertions (\ref{ref=Om^2})$\Rightarrow$(\ref{cl-ker})$\Rightarrow$(\ref{ex-ker}) are clear. We see (\ref{ex-ker})$\Rightarrow$(\ref{cl-ker}) by noting $\La\in\refl\La$. Looking at the kernel between maps of projectives, we deduce (\ref{cl-ker})$\Rightarrow$(\ref{ref=Om^2}).
	
	Finally, (\ref{ref=Om^2})$\Leftrightarrow$(\ref{grade1}) is \cite[Proposition 2.26]{ABr}, see also \cite[Proposition 1.6(a)]{AR96}.%Suppose that $\refl\La$ has kernels and let $K\xrightarrow{f} L\xrightarrow{g} M$ be a sequence in $\refl\La$ with $f=\ker g$. Since $\La\in\refl\La$, the sequence $0\to L\to M\to N$ is exact in $\mod\La$. This shows $\refl\subset\mod\La$ is closed under kernels. The remaining assertions are clear.
\end{proof}

\begin{Lem}\label{EX-coker}
Let $\La$ be a Noetherian ring. The following are equivalent.
\begin{enumerate}
\renewcommand\labelenumi{(\roman{enumi})}
\renewcommand\theenumi{\roman{enumi}}
\item\label{ex-cok} The category $\refl\La$ has cokernels.
\item\label{aast} For every $X\in\mod\La$, the evaluation map $X\to X^\aast$ induces an isomorphism $\Hom_\La(X^\aast,M)\xsimeq\Hom_\La(X,M)$ for all $M\in\refl\La$.
\end{enumerate}
In this case, the cokernel in $\refl\La$ is computed as the double dual of the cokernel in $\mod\La$.
\end{Lem}
\begin{proof}
	(\ref{aast})$\Rightarrow$(\ref{ex-cok})  Setting $M=\La$ we see that $X^\ast$ is reflexive. It follows that $X^\aast\in\refl\La$, and we deduce that $\refl\La$ has cokernels, which are computed as the double dual of the cokernel in $\mod\La$.
	
	(\ref{ex-cok})$\Rightarrow$(\ref{aast})  Let $X\in\mod\La$ and $M\in\refl\La$ and consider the map $\Hom_\La(X^\aast,M)\to\Hom_\La(X,M)$. Given a map $X\to M$, applying the double dual yields a morphism $X^\aast\to M$ since $M$ is reflexive, which shows surjectivity. It remains to prove injectivity. Since the cokernel of the map $X\to X^\aast$ is $\Ext^2_{\La^\op}(\Tr X,\La)$, this is equivalent to showing $\Hom_\La(\Ext^2_{\La^\op}(\Tr X,\La),M)=0$, that is, $\grade\Ext^2_{\La^\op}(Y,\La)\geq1$ for all $Y\in\mod\La^\op$. Notice that the map $(-)^\ast\colon\refl\La\leftrightarrow\refl\La^\op$ is a duality. Then $\refl\La$ having cokernels means $\refl\La^\op$ has kernels. By Lemma \ref{EX-ker}, this shows $\grade\Ext^2_{\La^\op}(Y,\La)\geq1$ for all $Y\in\mod\La^\op$.
\end{proof}

The isomorphism in Lemma \ref{EX-coker}(\ref{aast}) shows that the evaluation map $X\to X^\aast$ is initial among the maps from $X$ to reflexive modules. This justifies the following terminology.
\begin{Def}
Let $\La$ be a Noetherian ring satisfying the condition in Lemma \ref{EX-coker}.
We call the evaluation map $X\to X^\aast$, or the module $X^\aast$, the {\it reflexive hull} of $X$.
\end{Def}

We next discuss kernel-cokernel pairs in $\refl\La$.
\begin{Lem}\label{ses}
Let $\La$ be a Noetherian ring such that $\refl\La$ has kernels and cokernels.
A complex $L\xrightarrow{f} M\xrightarrow{g} N$ in $\refl\La$ is a kernel-cokernel pair if and only if the following conditions are satisfied.
\begin{enumerate}
\renewcommand\labelenumi{(\roman{enumi})}
\renewcommand\theenumi{\roman{enumi}}
\item $f=\ker g$ in $\mod\La$.
\item $C:=\Coker f\in\mod\La$ is torsion-free and $g$ factors as $M\twoheadrightarrow C\hookrightarrow C^{\ast\ast}\xsimeq N$.
\end{enumerate}
\end{Lem}
\begin{proof}
	By Lemma \ref{EX-ker}, the first condition is equivalent to $f=\ker g$ in $\refl\La$. By Lemma \ref{EX-coker}, the second condition is equivalent to $g=\coker f$.%If the two conditions are satified, then by \ref{kc}(\ref{k}) we have $f=\ker g$ in $\refl\La$, and by \ref{kc}(\ref{c}) that $g=\coker f$, hence the complex is a kernel-cokernel pair.
	%
	%Suppose conversely that $L\xrightarrow{f} M\xrightarrow{g} N$ is a kernel-cokernel pair in $\refl\La$. Since the kernels in $\refl \La$ are computed as the kernels in $\mod \La$, we have the first assertion, that is, the sequence $0\to L\xrightarrow{f}M\xrightarrow{g}N$ is exact in $\mod\La$. Then $C$ is the image (taken in $\mod\La$) of $g$ so this is torsion-free, and the description of the cokernels in $\refl\La$ as in \ref{kc}(\ref{c}) shows that $g$ factors as $M\twoheadrightarrow C\to C^{\ast\ast}\xsimeq N$.
\end{proof}

\subsubsection{Proof of Theorem \ref{max}, if part}
Now we assume that $\La$ is a Noetherian ring satisfying the two-sided $(2,2)$-condition. We shall prove that the category $\refl\La$ is quasi-abelian. We need a preparation on characterizations of admissible monomorphisms and epimorphisms.
\begin{Lem}\label{adm}
Let $\La$ be a Noetherian ring satisfying the two-sided $(2,2)$-condition and $f\colon L\to M$ be a morphism in $\refl \La$.
\begin{enumerate}
\item\label{mono} There is a kernel-cokernel pair $L\xrightarrow{f}M\to N$ in $\refl\La$ if and only if $f$ is a monomorphism in $\mod\La$ with torsion-free cokernel.
\item\label{epi} There is a kernel-cokernel pair $K\to L\xrightarrow{f}M$ in $\refl\La$ if and only if $\sgrade\Coker f\geq2$.%the cokernel has $\sgrade\geq2$.
%\item If $f$ in a monomorphism in $\mod \La$, then it is an admissible monomorphism in $\refl \La$ if and only if $\Coker f$ is torsion-free.
\end{enumerate}
\end{Lem}
\begin{proof}
	(1)  Since $\La$ satisfies the $(2,2)$-condition, we have by Lemma \ref{EX-ker} that $\refl\La\subset\mod\La$ is closed under kernels. Then the assertion follows immediately from Lemma \ref{ses}.
	
	(2)  Let $C:=\Im f$ in $\mod\La$. If there is a kernel-cokernel pair $K\to L\xrightarrow{f}M$, then we must have $\Coker f=\Ext^2_{\La^\op}(\Tr C,\La)$ by Lemma \ref{ses}. Since $\La^\op$ satisfies the $(2,2)$-condition, we have $\sgrade\Coker f\geq2$. 
	Suppose conversely that $\sgrade\Coker f\geq2$. Applying the $\Hom$-functor to a reflexive module to the exact sequence $0\to C\to M\to \Coker f\to0$ yields an isomorphism $\Hom_\La(M,-)\xsimeq\Hom_\La(C,-)$ on $\refl\La$ by Lemma \ref{easy}. It follows that the sequence $0\to\Hom_\La(M,-)\to\Hom_\La(L,-)\to\Hom_\La(K,-)$ is exact on $\refl\La$, that is, $f$ is a cokernel.
\end{proof}

We note an example of a non-admissible monomorphism or epimorphism.
\begin{Ex}
Let $\La=R$ be a commutative Noetherian normal domain. By Example \ref{2-Gor}, the ring $R$ is $2$-Gorenstein, thus satisfies the two-sided $(2,2)$-condition by Proposition \ref{comm}. Let $0\neq x\in R$ be a non-unit and consider the map $f\colon R\xrightarrow{x}R$.

%We claim that $f$ is a non-admissible monomorphism as well as a non-admissible epimorphism in $\refl R$.
\begin{enumerate}
\item The map $f$ is a monomorphism, but not the kernel of its cokernel. Indeed, Lemma \ref{EX-ker} shows that its kernel is $0$, thus it is a monomorphism in $\refl R$. Also its cokernel in $\refl R$ is $(R/(x))^{\ast\ast}=0$ by Lemma \ref{EX-coker}, so $\ker(\coker f)=\ker(R\to 0)=1_R\not\simeq f$. This shows that $f$ is not the kernel of its cokernel.
Noting that $\Coker f=R/(x)$ in $\mod R$ is not torsion-free, this is consistent with Lemma \ref{adm}(\ref{mono}). 
\item The map $f$ is a epimorphism in $\refl R$, but it is not the cokernel of its kernel. Indeed, as we have seen above, its cokernel in $\refl R$ is $0$ by Lemma \ref{EX-coker}, which shows that $f$ is an epimorphism in $\refl R$. On the other hand, the kernel is $0$ by Lemma \ref{EX-ker}, so $\coker(\ker f)=\coker(0\to R)=1_R\not\simeq f$. This shows that $f$ is not the cokernel of its kernel.
Noting that $\Coker f=R/(x)$ in $\mod R$ has grade $1$, this is consistent with Lemma \ref{adm}(\ref{epi}).  
\end{enumerate}
\end{Ex}

Before we prove the main theorem, we prepare the following observation which will be used later.
\begin{Lem}\label{torf}
Let $\La$ be a Noetherian ring satisfying the $(2,2)^\op$-condition, and let $X\in\mod\La$ be a torsion-free module. Then $\sgrade\Ext^1_{\La}(X,\La)\geq 2$.
\end{Lem}
\begin{proof}
	Since $X$ is torsion-free, it is a syzygy: $X=\Om X^\prime$ for some $X^\prime\in\mod\La$. This shows that $\Ext^1_{\La}(X,\La)=\Ext^2_{\La}(X^\prime,\La)$ has $\sgrade\geq2$ since $\La$ satisfies the $(2,2)^\op$-condition.
\end{proof}

Now we are ready to prove the `if' part of Theorem \ref{max}.
\begin{proof}%[Proof of \ref{max}, if part]
	%We verify the axioms in \ref{exdef}. The first one (\ref{EX1}) is clear.
	%
	%We next consider (\ref{EX2}). Recall that torsion-free modules are closed under extensions since $\La$ is $1$-Gorenstein. Then we deduce, in view of \ref{adm}(\ref{mono}), that a composite of admissible monomorphisms is again an admissible monomorphism. For epimorphisms, consider the composite $\bullet\xrightarrow{\, f\, }\bullet\xrightarrow{\, g\, }\bullet$ of admissible epimorphisms in $\refl\La$. Then we have an exact sequence $\Coker f\to \Coker gf\to \Coker g\to 0$ in $\mod\La$. Since the modules of strong grade $\geq2$ forms a Serre subcategory of $\mod\La$ by \ref{sgrade}, we see that $\Coker gf$ is also has strong grade $\geq2$.
	Since $\La$ satisfies the $(2,2)$-condition and the opposite, the condition in Lemma \ref{EX-ker}(\ref{grade1}) holds for $\La$ and $\La^\op$, therefore $\refl\La$ has kernels and cokernels.
	
	It remains to show that kernels (resp. cokernels) are closed under pull-backs (resp. push-outs). First consider the assertion for the push-out. Let $L\to M$ be an admissible monomorphism in $\refl\La$. Then by Lemma \ref{adm}(\ref{mono}) it extends to an exact sequence $0\to L\to M\to X\to0$ in $\mod\La$ with $X$ torsion-free. Let $L\to N$ be an arbitrary morphism in $\refl\La$ and consider the diagram
	\begin{equation}\label{eq9}
		\xymatrix{
			0\ar[r]&L\ar[r]\ar[d]&M\ar[r]\ar[d]& X\ar[r]\ar@{=}[d]& 0\\
			0\ar[r]&N\ar[r]\ar@{=}[d]&Y\ar[r]\ar[d]&X\ar[r]\ar[d]& 0 \\
			0\ar[r]&N\ar[r]&Y^\aast\ar[r]&Z\ar[r]&0 }
	\end{equation}
	where the upper left square is a push-out in $\mod\La$, and $Y\to Y^\aast$ is the reflexive hull of $Y$, so that pasting the left squares yields a push-out in $\refl\La$ by Lemma \ref{EX-coker}. We have to show that in the exact sequence of the last row, the module $Z\in\mod\La$ is torsion-free. Since $X$ is torsion-free and the category of torsion-free modules are closed under extension by Proposition \ref{AR}(\ref{Om^1}), we see by the middle exact sequence that $Y$ is also torsion-free. Then the lower right square yields an exact sequence
	\begin{equation}\label{eqXZ}
		\xymatrix{ 0\ar[r]&X\ar[r]&Z\ar[r]&\Ext^2_{\La^\op}(\Tr Y,\La)\ar[r]& 0 },
	\end{equation}
	for example, by the snake lemma applied to the bottom two rows of \eqref{eq9}.
	Applying $\Hom_\La(-,\La)$ yields an isomorphism $Z^\ast\xsimeq X^\ast$ by $\grade\Ext^2_{\La^\op}(\Tr Y,\La)\geq2$, hence $X^\aast\xsimeq Z^\aast$. Comparing this isomorphism with (\ref{eqXZ}), and using the snake lemma, we get an exact sequence
	\begin{equation}\label{eqinj}
		\xymatrix{ 0\ar[r]& \Ext^1_{\La^\op}(\Tr Z,\La)\ar[r]& \Ext^2_{\La^\op}(\Tr Y,\La)\ar[r]&  \Ext^2_{\La^\op}(\Tr X,\La)}
	\end{equation}
	since $X$ is torsion-free.
	On the other hand, the middle exact sequence in (\ref{eq9}) yields an exact sequence
	\begin{equation}\label{eqcl}
		\xymatrix{ 0\ar[r]& N^\aast\ar[r]& Y^\aast\ar[r]& X^\aast }.
	\end{equation}
	Indeed, dualizing the sequence yields an exact sequence $0\to X^\ast\to Y^\ast\to N^\ast\to C\to 0$, where $C$ is the cokernel of $Y^\ast\to N^\ast$, and is also a submodule of $\Ext^1_\La(X,\La)$. Since $X$ is torsion-free, we have $\sgrade\Ext^1_\La(X,\La)\geq2$ by Lemma \ref{torf}, and therefore $\grade C\geq2$. 
	%Since $X$ is torsion-free, it is a syzygy: $X=\Om X^\prime$ for some $X^\prime\in\mod\La$. This shows that $\Ext^1_\La(X,\La)=\Ext^2_\La(X^\prime,\La)$ has $\sgrade\geq2$, and therefore $\grade C\geq2$. 
	Then taking the dual once more yields the desired sequence.
	Now, comparing the middle exact sequence of (\ref{eq9}) and its double dual (\ref{eqcl}), we see (e.g.\,by the snake lemma) that the map $\Ext^2_{\La^\op}(\Tr Y,\La)\to \Ext^2_{\La^\op}(\Tr X,\La)$ is injective. We conclude by (\ref{eqinj}) that $\Ext^1_{\La^\op}(\Tr Z,\La)=0$, that is, $Z$ is torsion-free.
	
	Next we consider the pull-back. Let $f\colon L\to M$ be an admissible epimorphism in $\refl\La$. By Lemma \ref{adm}(\ref{epi}) it extends to an exact sequence
	\begin{equation}\label{eq4}
		\xymatrix@R=1.5mm@C=5mm{
			0\ar[rr] &&K\ar[rr]&&L\ar[rr]^-f\ar[dr]&& M\ar[rr]&& Y\ar[rr]&& 0\\
			&& 	&& &X\ar[ur]& && &&}
	\end{equation}
	in $\mod\La$ with $X:=\Im f$ and $\sgrade Y\geq2$. Now let $N\to M$ be an arbitrary morphism in $\refl\La$ and we want to form a pull-back of $f$ along it. Note that the pull-back in the category $\refl\La$ is computed as a pull-back in $\mod\La$ since $\refl\La\subset\mod\La$ is closed under kernels by Lemma \ref{EX-ker}. Let $Z$ be the image of $N\to M\to Y$, so that we get the diagram on the right below, with $X^\prime:=\Ker(N\to Z)$. It is easy to see that its left square is a pull-back.
	\[ 
	\xymatrix{
		0\ar[r]&K\ar@{=}[d]\ar[r]&L^\prime\ar[r]\ar[d]&X^\prime\ar[r]\ar[d]& 0\\
		0\ar[r]&K\ar[r]& L\ar[r]& X\ar[r]& 0}
	\qquad
	\xymatrix{
		0\ar[r]&X^\prime\ar[d]\ar[r]&N\ar[r]\ar[d]&Z\ar[r]\ar@{^(->}[d]& 0\\
		0\ar[r]&X\ar[r]& M\ar[r]& Y\ar[r]& 0}
	\]
	Also, consider the pull-back of the left half of the exact sequence in (\ref{eq4}) along $X^\prime\to X$, which yields a diagram on the left above.
	Now, pasting these pull-back squares yields the desired pull-back square (formed by $L^\prime$, $L$, $N$, and $M$) which lies in $\refl\La$, so the pull-back of the short exact sequence $K\to L\to M$ in $\refl\La$ is $K\to L^\prime\xrightarrow{f^\prime}N$. Now, we have $\Coker f^\prime=Z\subset Y$ and $\sgrade Y\geq2$, thus $\sgrade Z\geq2$. We conclude by Lemma \ref{adm}(\ref{epi}) that  $K\to L^\prime\xrightarrow{f^\prime}N$ is also a kernel-cokernel pair.
\end{proof}

\subsubsection{Proof of Theorem \ref{max}, only if part}
We prove the remaining half of Theorem \ref{max}.
\begin{proof}
Suppose that $\La$ is a Noetherian ring such that $\refl\La$ is quasi-abelian. Since $\refl\La$ has kernels and cokernels, the conditions in Lemma \ref{EX-ker} and Lemma \ref{EX-coker} hold.

Let $L\to M\to N$ be a conflation in $\refl\La$. Since $\refl\La\subset\mod\La$ is closed under kernels by Lemma \ref{EX-ker}, the sequence $0\to L\to M\to N$ is exact in $\mod\La$. We claim that the cokernel of the last map has $\sgrade\geq2$. Let $X=\Coker(M\to N)$ and $Y\subset X$ an arbitrary submodule. Pick a surjection $P\to Y$ from a projective module $P$ and lift to $P\to N$. We take a pull-back of $L\to M\to N$ along $P\to N$ to get a sequence on the first row below. This is a kernel-cokernel pair by our assumption that $\refl\La$ is quasi-abelian.
\[ \xymatrix{
	0\ar[r]&L\ar[r]\ar@{=}[d]&K\ar[r]\ar[d]&P\ar[r]\ar[d]&Y\ar[r]\ar@{^(->}[d]&0\\
	0\ar[r]&L\ar[r]&M\ar[r]&N\ar[r]&X\ar[r]&0} \]
Then, letting $Z$ be the image of $K\to P$, we see by Lemma \ref{EX-coker} that $Z\to P$ is a reflexive hull. Applying $\Hom_\La(-,\La)$ to the exact sequence $0\to Z\to P\to Y\to 0$, we get an exact sequence $0\to \Hom_\La(Y,\La)\to \Hom_\La(P,\La)\to \Hom_\La(Z,\La)\to\Ext^1_\La(Y,\La)\to0$, in which the middle map is an isomorphism by Lemma \ref{EX-coker}. This shows $\Hom_\La(Y,\La)=\Ext^1_\La(Y,\La)=0$, and hence $\sgrade X\geq2$.

We next show that for any $A\in\mod\La^\op$, the $\La$-module $\Ext^2_{\La^\op}(A,\La)$ appears as the cokernel of a deflation. %, which shows that it has $\sgrade\geq2$ by the previous claim, thus $\La^\op$ satisfies the $(2,2)$-condition.
Since $\Tr\colon\smod\La\leftrightarrow\smod\La^\op$ is a duality, it is enough to start from $X\in\mod\La$ and consider the space $\Ext^2_{\La^\op}(\Tr X,\La)$. Since it is the cokernel of $X\to X^\aast$, we may assume that $X$ is torsion-free by replacing $X$ by its image in $X^\aast$. Now, for a torsion-free module $X$, pick an exact sequence $0\to L\to M\to X\to 0$ in $\mod\La$ with $L,M\in\refl\La$. For example, we can take the obvious short exact sequence where $M$ is projective and $L=\Om X$. Then the sequence
\[ \xymatrix{ 0\ar[r]&L\ar[r]&M\ar[r]& X^\aast } \]
is a conflation in $\refl\La$, and the cokernel of the deflation is $\Ext^2_{\La^\op}(\Tr X,\La)$.

The previous claims yield that $\La$ satisfies the $(2,2)$-condition. Since $(\refl\La)^\op=\refl(\La^\op)$ is also a quasi-abelian category, we see that $\La^\op$ also satisfies the $(2,2)$-condition.
\end{proof}
\subsection{Consequences and remarks}
We list some immediate consequences of Theorem \ref{max} for some special classes of rings satisfying the two-sided $(2,2)$-conditions.

First, we obtain the following, in view of Example \ref{2-Gor}. Compare \cite{BB}, where they give this result for $\dim R=2$ using a categorical argument.
\begin{Cor}
Let $R$ be a commutative Noetherian normal domain. The category $\refl R$ is quasi-abelian.
\end{Cor}

For Artin algebras of global dimension $2$, we obtain a part of {\cite[Theorem 2.1]{Iy05}}.
\begin{Cor}
Let $\La$ be an Artin algebra of global dimension $2$. Then $\proj\La$ is quasi-abelian if and only if $\La$ satisfies the two-sided $(2,2)$-condition.
\end{Cor}

For a summary, we record the following characterization of conflations in $\refl\La$.
\begin{Prop}\label{confl}
Let $\La$ be a Noetherian ring satisfying the two-sided $(2,2)$-condition, and consider the category $\refl\La$ with the maximum exact structure.
For a complex $L\xrightarrow{f} M\xrightarrow{g} N$ in $\refl\La$, the following are equivalent.
\begin{enumerate}
\renewcommand\labelenumi{(\roman{enumi})}
\renewcommand\theenumi{\roman{enumi}}
\item\label{conf} This is a conflation in $\refl\La$.
\item\label{Ctorf} $f=\ker g$ in $\mod\La$, $C:=\Coker f\in\mod\La$ is torsion-free, and $g$ factors as $M\twoheadrightarrow C\hookrightarrow C^{\ast\ast}\xsimeq N$.
\item\label{Csgrade} $0\to L\xrightarrow{f} M\xrightarrow{g} N$ is exact in $\mod\La$ and $\sgrade\Coker g\geq2$.
\end{enumerate}
\end{Prop}
\begin{proof}
	(\ref{conf})$\Leftrightarrow$(\ref{Ctorf}) is nothing but Lemma \ref{ses}. We get (\ref{conf})$\Leftrightarrow$(\ref{Csgrade}) by applying Lemma \ref{adm}(\ref{epi}) to $g\colon M\to N$. 
%	If the two conditions are satified, then by \ref{kc}(\ref{k}) we have $f=\ker g$ in $\refl\La$, and by \ref{kc}(\ref{c}) that $g=\coker f$, hence the complex is a kernel-cokernel pair.
%	
%	Suppose conversely that $L\xrightarrow{f} M\xrightarrow{g} N$ is a kernel-cokernel pair in $\refl\La$. Since the kernels in $\refl \La$ are computed as the kernels in $\mod \La$, we have the first assertion, that is, the sequence $0\to L\xrightarrow{f}M\xrightarrow{g}N$ is exact in $\mod\La$. Then $C$ is the image (taken in $\mod\La$) of $g$ so this is torsion-free, and the description of the cokernels in $\refl\La$ as in \ref{kc}(\ref{c}) shows that $g$ factors as $M\twoheadrightarrow C\to C^{\ast\ast}\xsimeq N$.
%	%	, with $C:=\Coker f$. It remains to show that $C$ is torsion-free, that is, the map $C\to C^{\ast\ast}$ is injective. By the factorization of $g$, we have a commutative diagram
%	%	\[ \xymatrix{
%		%		0\ar[r]&L\ar[r]^-f\ar@{=}[d]&M\ar[r]\ar@{=}[d]&C\ar[r]\ar[d]&0\\
%		%		0\ar[r]&L\ar[r]^-f&M\ar[r]^-g&N } \]
%	%	in $\mod\La$ whose rows are exact. It follows that $C\to N$ has to be injective, hence so is $C\to C^{\ast\ast}$.
\end{proof}

We also note the following information where (ii)$\Leftrightarrow$(iii) is a special case of \cite[Proposition 2.4]{Iy03}.
\begin{Rem}
If $\La$ satisfies the two-sided $(2,2)$-condition, the following are equivalent for $X\in\mod\La$.
\begin{enumerate}
\renewcommand\labelenumi{(\roman{enumi})}
\renewcommand\theenumi{\roman{enumi}}
\item $X\simeq\Ext^2_{\La^\op}(A,\La)$ for some $A\in\mod\La^\op$.
\item $\sgrade X\geq2$.
\item $\grade X\geq2$.
\end{enumerate}
Indeed, we only have to show that $\grade X\geq2$ implies $X\simeq\Ext^2_{\La^\op}(A,\La)$ for some $A\in\mod\La^\op$. Pick an exact sequence $0\to L\to M\to N\to X\to 0$ in $\mod\La$ with $L,M,N\in\refl\La$. For example, one can take $M\to N$ as a projective presentation of $X$ and $L$ as its kernel. Let $C$ be the image of $M\to N$. Since $\grade X\geq2$, applying $\Hom_\La(-,\La)$ to the exact sequence $0\to C\to N\to X\to 0$ shows that we have $N^\ast\xsimeq C^\ast$, and hence $C^\aast\xsimeq N$. It follows that $C\to N$ is a reflexive hull, and hence the exact sequence $0\to C\to N\to X\to0$ in $\mod\La$ is isomorphic to $0\to C\to C^\aast\to \Ext^2_{\La^\op}(\Tr C,\La)\to0$, in particular $X\simeq\Ext^2_{\La^\op}(\Tr C,\La)$.
%	
%Slightly weakly, suppose that $\La$ is $1$-Gorenstein and $\grade\Ext^2_{\La^\op}(A,\La)\geq2$ for all $A\in\mod\La^\op$. Then only (i) and (iii) are equivalent.
\end{Rem}

\subsection{Abelian categories}
We have seen that the category $\refl\La$ is quasi-abelian if and only if $\La$ satisfies the two-sided $(2,2)$-condition. We further give a characterization when $\refl\La$ is abelian.
\begin{Thm}\label{abelian}
Let $\La$ be a Noetherian ring. Then $\refl\La$ is abelian if and only if $\ddim\La\geq2$.
\end{Thm}

\begin{Rem}
It should be noted that $\ddim\La\geq2$ implies $\La$ is Artinian \cite[Proposition 7]{IwS}. Therefore, when $\La$ is finitely generated over its center, the situation of Theorem \ref{abelian} is the same as the situation of Morita-Tachikawa correspondence (see Theorem \ref{MT}).
\end{Rem}

%Let $\La$ be a Noetherian ring. Recall from Definition \ref{n} that a $\La$-module $M$ is torsion-free if it satisfies the following equivalent conditions.
%\begin{itemize}
%\item $M$ is a submodule of a projective module.
%\item The natural map $M\to M^\aast$ is injective.
%\end{itemize}
%On the other hand, we say that a $\La$-module $C$ is {\it torsion} if it satisfies the following equivalent conditions.
%%Let $\La$ be an arbitrary Noetherian ring. We say that a $\La$-module $C$ is {\it torsion} if it satisfies the following equivalent conditions.
%\begin{itemize}
%\item $\Hom_\La(C,\La)=0$.
%\item The natural map $C\to C^\aast$ is $0$.
%%\item $\Hom_\La(C,L)=0$ for all torsion-free $L\in\mod\La$.
%\end{itemize}
%Suppose that $\La$ satisfies the $(2,1)^\op$-condition. (For example, this is implied by the $(2,2)^\op$-condition.) Then the pair $(\T,\F)$ with $\T$ the class of torsion modules and $\F$ the class of torsion-free modules form a torsion pair in the sense that it satisfies the following conditions.
%\begin{enumerate}
%\item We have $\Hom_\La(T,F)=0$ for each $T\in\T$ and $F\in\F$.
%\item For each $M\in\mod\La$, there is an exact sequence $0\to T\to M\to F\to 0$ with $T\in\T$ and $F\in\F$.
%\end{enumerate}
%In this case, the sequence in (2) is obtained by setting $T$ (resp. $F$) to be the kernel (resp. the image) of the evaluation map $M\to M^\aast$.

Let $\La$ be an arbitrary Noetherian ring. We say that a $\La$-module $C$ is {\it torsion} if it satisfies the following equivalent conditions.
\begin{itemize}
	\item $\Hom_\La(C,\La)=0$.
	\item The natural map $C\to C^\aast$ is $0$.
	%\item $\Hom_\La(C,L)=0$ for all torsion-free $L\in\mod\La$.
\end{itemize}
The following observation is crucial.
\begin{Lem}\label{tors}
Let $\La$ be an arbitrary Noetherian ring satisfying the $(2,1)^\op$-condition. Then $\ddim\La\geq2$ if and only if $\sgrade C\geq2$ for every torsion $\La$-module $C$.
\end{Lem}
\begin{proof}
	By the above equivalent conditions and the Auslander-Bridger sequence $0\to\Ext^1_{\La^\op}(\Tr C,\La)\to C\to C^\aast$, we see that every torsion module is of the form $\Ext^1_{\La^\op}(X,\La)$ for some $X\in\mod\La^\op$. Conversely, the $(2,1)^\op$-condition means that every module of the form $\Ext^1_{\La^\op}(X,\La)$ is torsion.
	We then obtain the desired result by recalling that $\ddim\La\geq2$ is nothing but the $(1,2)$-condition on $\La$.
\end{proof}

Now we are ready to prove the main result of this subsection.
\begin{proof}[Proof of Theorem \ref{abelian}]
	Note first that the condition $\ddim\La\geq2$ is nothing but the $(1,2)$-condition. By left-right symmetry of dominant dimension \cite{Ho}, this is equivalent to the two-sided $(1,2)$-condition. In particular, it satisfies the two-sided $(2,2)$-condition.
	Note also that abelian categories are exactly the quasi-abelian categories where every monomorphism (resp. epimorphism) is admissible (with respect to the maximum exact structure).
	
	\medskip
	Suppose that $\ddim\La\geq2$, so that $\sgrade\Ext^1_{\La^\op}(X,\La)\geq2$ for all $X\in\mod\La^\op$. We want to prove that every monomorphism (resp. epimorphism) in $\refl\La$ is admissible.
	
	Let $L\to M$ be a monomorphism in $\refl\La$, which is just a monomorphism in $\mod\La$. In view of Lemma \ref{adm}(\ref{mono}), we have to show that the cokernel $C$ in $\mod\La$ is torsion-free, in other words, that the torsion part $T:=\Ker(C\to C^\aast)=\Ext^1_{\La^\op}(\Tr C,\La)$ is $0$. By assumption we have $\grade T\geq2$, thus $\Hom_\La(T,N)=\Ext^1_\La(T,N)=0$ for all $N\in\refl\La$ by Lemma \ref{easy}. Then, applying $\Hom_\La(T,-)$ to the exact sequence $0\to L\to M\to C\to 0$ yields an exact sequence
	\[ \xymatrix@!R=1mm{
		\Hom_\La(T,M)\ar[r]\ar@{=}[d]&\Hom_\La(T,C)\ar[r]&\Ext^1_\La(T,L)\ar@{=}[d]\\
		0&&0&, } \]
	which shows $\Hom_\La(T,C)=0$, hence $T=0$.
	
	We next consider an epimorphism $M\to N$ in $\refl\La$. Extend it to an exact sequence $M\to N\to C\to 0$ in $\mod\La$. Applying $\Hom_\La(-,L)$ for each $L\in\refl\La$, we have an exact sequence $0\to\Hom_\La(C,L)\to \Hom_\La(N,L)\to \Hom_\La(M,L)$, which shows that $M\to N$ is an epimorphism in $\refl\La$ if and only if $\Hom_\La(C,-)=0$ on $\refl\La$. Since each reflexive module is a submodule of a projectie module, this is equivalent to $\Hom_\La(C,\La)=0$, in other words, $C$ is torsion. Then by Lemma \ref{tors} and Lemma \ref{adm}(\ref{epi}), we see that $M\to N$ is admissible.
	
	\medskip
	Now suppose conversely that $\refl\La$ is abelian. In particular, it is quasi-abelian, thus $\La$ satisfies the $(2,2)^\op$-condition by Theorem \ref{max}.
	Then we may apply Lemma \ref{tors}, so we have to show that $\sgrade C\geq2$ for all torsion $\La$-modules $C$. Let $C$ be such a module and consider a projective presentation $P_1\xrightarrow{f}P_0\to C\to 0$. Then $f\colon P_1\to P_0$ is an epimorphism in $\refl\La$ since $C$ is torsion, which extends to a conflation $\Ker f\to P_1\to P_0$ by our assumption that $\refl\La$ is abelian. By Proposition \ref{confl}, we conclude that the cokernel $C$ of the deflation has $\sgrade\geq2$.
\end{proof}

\section{Artin algebras}
We apply the results in the previous section to finite dimensional over a field $k$. Let $\La$ be such an algebra and
\[ \xymatrix{ 0\ar[r]&\La\ar[r]&I^0\ar[r]&I^1\ar[r]&I^2\ar[r]&\cdots } \]
be the minimal injective resolution in $\mod\La$.
Then,
\begin{itemize}
\item the $(1,2)$-condition says $\pd I^0=\pd I^1=0$, that is, $\ddim\La\geq2$,
\item the $(2,2)$-condition says $\pd I^0\leq1$ and $\pd I^1\leq1$.
\end{itemize}

Let us recall the following classical result on finite dimensional algebras of $\ddim\geq2$. We say that pairs $(\Si,A)$ and $(\Si^\prime,A^\prime)$ consisting of a finite dimensional algebra $\Si$ (resp. $\Si^\prime$) and $A\in\mod\Si$ (resp. $A^\prime\in\mod\Si^\prime$) are {\it Morita equivalent} if there is an equivalence $\mod\Si\simeq\mod\Si^\prime$ restricting to an equivalence $\add A\simeq\add A^\prime$. Also, we say that a module $A\in\mod\Si$ is a {\it generator} (resp. a {\it cogenerator}) if $\Si\in\add A$ (resp. $D\Si\in\add A$).
\begin{Thm}[Morita-Tachikawa correspondence]\label{MT}
There exists a bijection between the following.
\begin{enumerate}
\renewcommand\labelenumi{(\roman{enumi})}
\renewcommand\theenumi{\roman{enumi}}
\item The set of Morita equivalence classes of pairs $(\Si,A)$, where $\Si$ is a finite dimensional algebra and $A\in\mod\Si$ a generator-cogenerator.
\item The set of Morita equivalence classes of finite dimensional algebras $\La$ with $\ddim\La\geq2$.
\end{enumerate}
The correspondence is given by $(\Si,A)\mapsto\End_\Si(A)$.
\end{Thm}

Let $\La$ be a finite dimensional algebra of $\ddim\geq2$. Then by Theorem \ref{abelian} the category $\refl\La$ is abelian. On the other hand, the Morita-Tachikawa correspondence above says that there is a finite dimensional algebra $\Si$ and a generator-cogenerator $A\in\mod\Si$ such that $\La=\End_\Si(A)$. Our interpretation of Theorem \ref{MT} in terms of Theorem \ref{abelian} is the following, which says that the abelian category $\refl\La$ is nothing but the module category over the original algebra $\Si$.
\begin{Prop}\label{mt}
We have an equivalence $\refl\La\simeq\mod\Si$.
\end{Prop}
\begin{proof}
	By $\ddim\La\geq2$, there exists a pair $(\Si,A)$ of a finite dimensional algebra $\Si$ and a generator-cogenerator $A\in\mod\Si$ such that $\La=\End_\Si(A)$ by Theorem \ref{MT}. We show that the functor 
	\[ \xymatrix{ \Hom_\Si(A,-)\colon\mod\Si\ar[r]&\refl\La } \]
	is an equivalence.
	
	We first note that this functor is well-defined, that is, the functor takes values in $\refl\La$. Indeed, take an injective resolution $0\to X\to I^0\to I^1$ of $X\in\mod\Si$, which yields an exact sequence $0\to\Hom_\Si(A,X)\to\Hom_\Si(A,I^0)\to\Hom_\Si(A,I^1)$ in $\mod\La$. Since $A$ is a cogenerator, the last two terms lie in $\proj\La$, thus $\Hom_\Si(A,X)$ is a second syzygy, hence is reflexive by Lemma \ref{EX-ker}.
	
	Next, it is well-known that $A$ being a generator implies that the functor $\Hom_\Si(A,-)$ is fully faithful (see e.g.\! the proof of Theorem \ref{Morita} below).
	
	Finally, we show that the functor is dense. If $M\in\refl\La$ then there is an exact sequence $0\to M\to P^0\to P^1$ in $\mod\La$ with $P^0,P^1\in\proj\La$. Since $\Hom_\Si(A,-)\colon\add A\to\proj\La$ is an equivalence, the map $P^0\to P^1$ comes from a map $A^0\to A^1$ in $\add A\subset\mod\Si$. Letting $B\in\mod\Si$ be its kernel, we see that $M=\Hom_\Si(A,B)$.
\end{proof}

For example, one can take (higher) Auslander algebras as algebras \cite{Iy07a,Iy07b} of dominant dimension $\geq2$. 
\begin{Ex}
%\begin{enumerate}
%\item Let $\La$ be an Auslander algebra, that is one has $\gldim\La\leq 2\leq\domdim\La$. Then there is a finite dimensional algebra $\Si$ of finite representation type with $\add M=\mod\Si$ such that $\La=\End_\Si(M)$. By \ref{mt} we have 
%\item More generally, 
Let $n\geq1$ and let $\La$ be an $n$-Auslander algebra, that is one has $\gldim\La\leq n+1\leq\domdim\La$. Let $(\Si,A)$ be the pair consisting of a finite dimensional algebra $\Si$ and $A\in\mod\Si$ an $n$-cluster tilting object. By Proposition \ref{mt} we see that $\mod\Si\simeq\refl\La$, which is a description of the original module category $\mod\Si$ in terms of the (higher) Auslander algebra $\La$.
\end{Ex}
	
As another instance of an interpretation of our result, we recall the following result on finite dimensional algebras satisfying the two-sided $(2,2)$-condition.
\begin{Thm}[Auslander correspondence for torsion(-free) classes, \cite{Iy05}]\label{Ators}
Let $\La$ be a finite dimensional algebra. There exists a bijection between the following.
\begin{enumerate}
\renewcommand\labelenumi{(\roman{enumi})}
\renewcommand\theenumi{\roman{enumi}}
\item The set of equivalence classes of categories $\C$ satisfying that there is a finite dimensional algebra $\Si$ such that $\C$ is equivalent to a torsion-free class in $\mod\Si$.
\item The set of Morita equivalence classes of finite dimensional algebras $\La$ of global dimension $\leq2$ satisfying the two-sided $(2,2)$-condition.
\end{enumerate}
\end{Thm}

Applying Theorem \ref{max} to such Aulander algebras of torsion(-free) classes, we get the following which is a part of \cite[Theorem 2.1]{Iy05}.
\begin{Prop}
Let $\Si$ be a finite dimensional algebra and let $M\in\mod\Si$ such that $\add M\subset\mod\Si$ is a torsion-free class. Then $\add M$ is quasi-abelian.
\end{Prop}
\begin{proof}
	Let $\La=\End_\Si(M)$. By Theorem \ref{Ators}, the algebra $\La$ has global dimension $\leq 2$ and satisfies the two-sided $(2,2)$-condition. We have $\refl\La\simeq\proj\La\simeq\add M$ by $\gldim\La\leq2$, and this category is quasi-abelian by Theorem \ref{max}.
\end{proof}

In view of the above results, it would be interesting to ask the following.
\begin{Qs}
Let $\La$ be a finite dimensional algebra satisfying the two-sided $(2,2)$-condition. Does there exist a finite dimensional algebra $\Si$ such that the category $\refl\La$ has a nice representation theoretic meaning in terms of $\Si$?
\end{Qs}
When $\La$ has dominant dimension $\geq2$, then $\refl\La$ is nothing but a module category over some $\Si$ by Proposition \ref{mt}. Also, when $\La$ has global dimension $\leq2$, then $\refl\La$ is a torsion(-free) class in a module category over some $\Si$ by Theorem \ref{Ators}. The above question asks whether one can find such a $\Si$ for general $\La$.

\section{Non-maximum exact structures}
Let $\La$ be a Noetherian ring satisfying the two-sided $(2,2)$-condition. We have seen in Theorem \ref{max} that the category $\refl\La$ is quasi-abelian, so it has a maximum exact structure. The aim of this section is to give a systematic construction of exact structures on $\refl\La$.

For a Noetherian ring $\La$ we set
\[ \D:=\{ X\in\mod\La \mid \sgrade X\geq2\}, \]
which is a Serre subcategory of $\mod\La$ (see Proposition \ref{sgrade}).
\begin{Thm}\label{subset}
Let $\La$ be a Noetherian ring satisfying the two-sided $(2,2)$-condition. Consider the following two classes.
\begin{enumerate}
\renewcommand\labelenumi{(\roman{enumi})}
\renewcommand\theenumi{\roman{enumi}}
\item\label{se} Serre subcategories of $\mod\La$ contained in $\D$.
\item\label{ex} Exact structures on $\refl\La$.
\end{enumerate}
%Then we have well-defined correspondences given as follows:
%\begin{itemize}
%\item From {\rm(\ref{ex})} to {\rm(\ref{se})}: For an exact structure $\S$ (the class of conflations), assign the subcategory $\{ \Coker g\in\mod\La \mid L\xrightarrow{f}M\xrightarrow{g}N \text{ is in } \S\}$.
%\item From {\rm(\ref{se})} to {\rm(\ref{ex})}: For a Serre subcategory $\A\subset\D$, declare the sequence $L\xrightarrow{f}M\xrightarrow{g}N$ in $\refl\La$ to be a conflation if and only if $0\to L\xrightarrow{f}M\xrightarrow{g}N$ is exact in $\mod\La$ and $\Coker g\in\A$.
%\end{itemize}
%Moreover, the composite {\rm(\ref{se})} $\to$ {\rm(\ref{ex})} $\to$ {\rm(\ref{se})} is the identity on the class {\rm(\ref{se})}.
For a Serre subcategory $\A\subset\D$, define the class of sequences in $\refl\La$ by 
\[ \S(\A):=\{ L\xrightarrow{f} M\xrightarrow{g} N \text{ in }\refl\La \mid 0\to L\xrightarrow{f}M\xrightarrow{g}N \text{ is exact in } \mod\La \text{ and } \Coker g\in \A\}. \]
Then this gives a well-defined injection from {\rm (\ref{se})} to {\rm (\ref{ex})}.
\end{Thm}

Let us show that this correspondence from (\ref{se}) to (\ref{ex}) is well-defined. %Recall from \ref{max} that
\begin{Prop}\label{se to ex}
Let $\A\subset\D$ a Serre subcategory of $\mod\La$. Then the class $\S(\A)$
%\[ \S(\A):=\{ L\xrightarrow{f} M\xrightarrow{g} N \text{ in }\refl\La \mid 0\to L\xrightarrow{f}M\xrightarrow{g}N \text{ is exact in } \mod\La \text{ and } \Coker g\in \A\} \]
forms an exact structure on $\refl\La$.
\end{Prop}

We start with a preparation which is a generalization of Lemma \ref{adm}.
\begin{Lem}\label{admA}
Let $\A\subset\D$ be a Serre subcategory of $\mod\La$ and let $f\colon L\to M$ be a morphism in $\refl\La$.
\begin{enumerate}
\item There is a sequence $L\xrightarrow{f} M\to N$ in $\S(\A)$ if and only if $f$ is a monomorphism in $\mod\La$, $X:=\Coker f\in\mod\La$ is torsion-free, and satisfies $\Ext^2_{\La^\op}(\Tr X,\La)\in\A$.
\item There is a sequence $K\to L\xrightarrow{f} M$ in $\S(\A)$ if and only if $\Coker f\in\A$.
\end{enumerate}
%and let $0\to L\xrightarrow{f} M\xrightarrow{g} N$ be an exact sequence in $\mod\La$ with terms in $\refl\La$.
\end{Lem}
\begin{proof}
	(1)  If $f$ extends to a sequence $L\xrightarrow{f} M\xrightarrow{g} N$ in $\S(\A)$, then $f$ is a monomorphism in $\mod\La$ and $A:=\Coker g$ lies in $\A$. Consider the exact sequence $0\to X\to N\to A\to0$. This first shows that $X$, being a submodule of $N$, is torsion-free. Also, since $A\in\A\subset\D$ has $\sgrade\geq2$, we see that $X\to N$ is a reflexive hull, hence $\Ext^2_{\La^\op}(\Tr X,\La)=A\in\A$.
	
	Suppose conversely that $X=\Coker f$ is torsion-free and $\Ext^2_{\La^\op}(\Tr X,\La)\in\A$. Then the Auslander-Bridger sequence yields an exact sequence $0\to X\to X^\aast\to \Ext^2_{\La^\op}(\Tr X,\La)\to 0$. This shows that the sequence $0\to L\xrightarrow{f}M\to X^\aast$ lies in $\S(\A)$.
	
	(2)  Since $\refl\La\subset\mod\La$ is closed under kernels by Lemma \ref{EX-ker}, this is clear from the definition of $\S(\A)$.
\end{proof}
	
Now we are able to prove well-definedness. %of the correspondence from (\ref{se}) to (\ref{ex}) in Theorem \ref{subset}.
\begin{proof}[Proof of Proposition \ref{se to ex}]
	We have to check the axioms in Definition \ref{exdef}. 
	
	(i)  Clearly, the class $\S(\A)$ is closed under isomorphisms and contains split exact sequences.
	
	(ii)  We show that admissible monomorphisms are closed under compositions. Let $L\xrightarrow{f}M$ and $M\xrightarrow{g}N$ be admissible monomorphisms. By Lemma \ref{admA}(\ref{mono}), we know that $X:=\Coker f$ and $Z:=\Coker g$ are torsion-free, and that $\Ext^2_{\La^\op}(\Tr X,\La), \Ext^2_{\La^\op}(\Tr Z,\La)\in\A$. If we put $Y:=\Coker(gf)$, the diagram
	\[ \xymatrix{
		L\ar[r]^-f\ar@{=}[d]&M\ar[r]\ar[d]^-g& X\ar[d] \\
		L\ar[r]& N\ar[r]\ar[d]& Y\ar[d]\\
		&Z\ar@{=}[r]&Z } \]
	of conflations shows $Y$ is also torsion-free by Proposition \ref{AR}(\ref{Om^1}). It remains to prove $\Ext^2_{\La^\op}(\Tr Y,\La)\in\A$. Dualizing the short exact sequence $0\to X\to Y\to Z\to 0$, we get $0\to Z^\ast\to Y^\ast\to X^\ast$, and the cokernel $C$ of the last map is a submodule of $\Ext^1_{\La}(Z,\La)$. Since $Z$ is torsion-free, we see that $\sgrade\Ext^1_{\La}(Z,\La)\geq2$ by Lemma \ref{torf}, hence $\grade C\geq2$. It follows that the doubly dualized sequence $0\to X^\aast\to Y^\aast\to Z^\aast$ is exact. Then the commutative diagram
	\[ \xymatrix{
		0\ar[r]& X\ar[r]\ar[d]& Y\ar[r]\ar[d]& Z\ar[r]\ar[d]& 0\\
		0\ar[r]& X^\aast\ar[r]& Y^\aast\ar[r]& Z^\aast } \]
	yields an exact sequence
	\[ \xymatrix{ 0\ar[r]&\Ext^2_{\La^\op}(\Tr X,\La)\ar[r]&\Ext^2_{\La^\op}(\Tr Y,\La)\ar[r]&\Ext^2_{\La^\op}(\Tr Z,\La) } \]
	of the cokernels of the vertical maps, for example, by the snake lemma. Since $\A$ is a Serre subcategory, this shows that the middle term $\Ext^2_{\La^\op}(\Tr Y,\La)$ also lies in $\A$.
	
	(ii)$^\op$ We next show that admissible epimorphisms are closed under compositions. If $f\colon L\to M$ and $g\colon M\to N$ are admissible epimorphisms, it is easy to see that we have an exact sequence $\Coker f\to \Coker(gf)\to \Coker g$. Since $\A$ is a Serre subcategory, $\Coker f$ and $\Coker g$ lying in $\A$ implies $\Coker(gf)\in\A$.
	
	(iii)  We show that admissible monomorphisms are closed under push-outs. Let $f\colon L\to M$ be an admissible monomorphism and $g\colon L\to N$ an arbitrary morphism in $\refl\La$. As in the proof of Theorem \ref{max}, the push-out of $f$ along $g$ in $\refl\La$ is computed as the reflexive hull of the push-out in $\mod\La$, and we get the following diagram identical to (\ref{eq9}).
	\[ \xymatrix{
			0\ar[r]&L\ar[r]^-f\ar[d]_-g&M\ar[r]\ar[d]& X\ar[r]\ar@{=}[d]& 0\\
			0\ar[r]&N\ar[r]\ar@{=}[d]&Y\ar[r]\ar[d]&X\ar[r]\ar[d]& 0 \\
			0\ar[r]&N\ar[r]^-h&Y^\aast\ar[r]&Z\ar[r]&0 } \]
	We want to show that $h$ is an admissible monomorphism, which is to say, by Lemma \ref{admA}(\ref{mono}), that $Z$ is torsion-free and $\Ext^2_{\La^\op}(\Tr Z,\La)\in\A$. We know from the proof of Theorem \ref{max} that $Z$ is indeed torsion-free. Also, the last two rows yield an exact sequence $0\to X\to Z\to\Ext^2_{\La^\op}(\Tr Y,\La)\to 0$, hence $X^\aast\xsimeq Z^\aast$ by $\grade\Ext^2_{\La^\op}(\Tr Y,\La)\geq2$. Comparing these we get (for example by the snake lemma) a surjection $\Ext^2_{\La^\op}(\Tr X,\La)\twoheadrightarrow \Ext^2_{\La^\op}(\Tr Z,\La)$. Since $\A$ is a Serre subcategory, we get $\Ext^2_{\La^\op}(\Tr Z,\La)\in\A$ from $\Ext^2_{\La^\op}(\Tr X,\La)\in\A$.
	
	(iii)$^\op$  We finally prove that the admissible epimorphisms are closed under pull-backs, but this is done analogously to the proof of Theorem \ref{max}, see the last paragraph there.
\end{proof}

To prove injectivity, we note the following reconstruction.
\begin{Lem}\label{ex to se}
For an exact structure $\S$ on $\refl\La$, define a subcategory of $\mod\La$ by $\A(\S):=\{ \Coker g\in\mod\La \mid L\xrightarrow{f}M\xrightarrow{g}N \text{ is in } \S\}$. Then for each Serre subcategory $\A\subset\D$, we have $\A(\S(\A))=\A$.
\end{Lem}

%We need the following lemma.
%\begin{Lem}\label{proj}
%Let $\S$ be an exact structure on $\refl\La$. Then for each $A\in\A(\S)$ there is a conflation $M\xrightarrow{f} Q\xrightarrow{g} P$ in $\S$ with $P,Q\in\proj\La$ such that $\Coker g=A$.
%\end{Lem}
%\begin{proof}
%	If $A\in\A(\S)$ there is a conflation $L\to M\to N$ in $\refl\La$ which is part of an exact sequence $0\to L\to M\to N\to A\to 0$ in $\mod\La$. First, pick a surjection $P\to A$ with $P\in\proj\La$ and lift it to $P\to N$, which yields the commutative diagram below by taking a pull-back of $M\rightarrow N\leftarrow P$.
%	\[ \xymatrix{
%		0\ar[r]&L\ar[r]\ar@{=}[d]&K\ar[r]\ar[d]&P\ar[r]\ar[d]&A\ar@{=}[d]\ar[r]&0\\
%		0\ar[r]&L\ar[r]&M\ar[r]&N\ar[r]& A\ar[r]&0 } \]
%	Then the first row, being the pull-back of the conflation in the second row, is a conflation.
%	
%	Next, pick a surjection $Q\twoheadrightarrow K$ with $Q\in\proj\La$, which yields the following diagram.
%	\[ \xymatrix{
%		0\ar[r]&\Om^2A\ar[r]\ar[d]&Q\ar[r]\ar@{->>}[d]&P\ar[r]\ar@{=}[d]&A\ar@{=}[d]\ar[r]&0\\
%		0\ar[r]&L\ar[r]&K\ar[r]&P\ar[r]& A\ar[r]&0 } \]
%	Now, the map $Q\to P$ is a deflation since it is a composite $Q\to K\to P$ of deflations, and therefore the first row is also a conflation.
%\end{proof}

\begin{proof}%[Proof of Proposition \ref{ex to se}]
	By the definition of $\S(\A)$, we have $\A(\S(\A))\subset\A$. Conversely, suppose that $A\in\A$. Then we have an exact sequence $0\to \Om^2A\to P_1\to P_0\to A\to 0$ in $\mod\La$ with $P_0,P_1\in\proj\La$. We have $\Om^2A\in\refl\La$ by Lemma \ref{EX-ker}, thus $\Om^2\to P_1\to P_0$ is an $\S(\A)$-conflation. This shows $A\in\A(\S(\A))$.
\end{proof}

Now we summarize the proof of Theorem \ref{subset}.
\begin{proof}[Proof of Theorem \ref{subset}]
	The map is well-defined by Proposition \ref{se to ex}, and injective by Lemma \ref{ex to se}.
	%The maps are well-defined by Proposition \ref{se to ex} and Proposition \ref{ex to se}. To prove that the composite {\rm(\ref{se})} $\to$ {\rm(\ref{ex})} $\to$ {\rm(\ref{se})} is the identity, we have to check that every object in $\A$ appears as the cokernel of an $\S(\A)$-deflation. We can indeed take it, for example, as $0\to \Om^2A\to P_1\to P_0\to A\to 0$ for each $A\in\A$ with $P_0,P_1\in\proj\La$.
\end{proof}

The following example shows that the above correspondences are not bijections. In fact, while the class (\ref{ex}) is clearly invariant under the reflexive equivalence (that is, having the equivalent category of reflexive modules), the class (\ref{se}) does depend on it.
\begin{Ex}
Let $k$ be an algebraically closed field of characteristic $0$, and let $S=k[[x,y]]$. Let $G$ be a finite small subgroup of $\GL_2(k)$, and let $R=S^G$ the invariant ring. Then we have $S\ast G\xsimeq\End_R(S)$ and a reflexive equivalence
\[ \xymatrix{ \Hom_R(S,-)\colon \refl R\ar[r]^-\simeq&\refl S\ast G}. \]
Note that since $R$ and $S\ast G$ are of dimension $2$, the category $\D$ is nothing but the finite length modules over these rings, and the Serre subcategories of $\fl\La$ are classified by the subsets of simple $\La$-modules. 
\end{Ex}

%%The following question would be interesting!!
%\begin{Qs}
%For which ring $\La$ are the above correspondences bijections?
%\end{Qs}

\begin{Rem}
This should be compared with the results in \cite{E}, which gives a classification of exact structures on an additive category $\E$ in terms of its functor category $\mod\E$. Note that Theorem \ref{subset} above describes the exact structures on $\E=\refl\La$ in terms of $\mod\La$, not $\mod\E$, so the results are independent in this sense.
\end{Rem}

\section{Morita theorem}
\subsection{Statements and examples}
We establish a Morita theorem characterizing the category of reflexive modules among quasi-abelian categories. Let us prepare one notion on exact categories.
\begin{Def}
Let $\E$ be an exact category. An object $M\in\E$ is called
\begin{itemize}
\item a {\it generator} if for any $X\in\E$ there is an admissible epimorphism $N\to X$ with $N\in\add M$,
\item a {\it cogenerator} if for any $X\in \E$ there is an admissible monomorphism $X\to N$ with $N\in\add M$,
\item a {\it generator-cogenerator} if it is both a generator and a cogenerator.
\end{itemize}
\end{Def}

Our Morita theorem for quasi-abelian categories is the following.
\begin{Thm}\label{Morita}
Let $\C$ be a quasi-abelian category. Suppose that there is a generator-cogenerator $M\in\C$ such that $\La:=\End_\C(M)$ is two-sided Noetherian. Then $\La$ satisfies the two-sided $(2,2)$-condition and we have an equivalence
\[ \xymatrix{ \Hom_\C(M,-)\colon\C\ar[r]^-\simeq&\refl\La }. \]
\end{Thm}

Consequently we obtain the following characterization of reflexive equivalence. 
We refer to \cite[Proposition 2.4]{IR} for an analogous result for module-finite algebras over commutative Noetherian normal domains.
\begin{Thm}\label{refeq}
Let $\La$ be a Noetherian ring satisfying the two-sided $(2,2)$-condition. The following are equivalent for a Noetherian ring $\Ga$.
\begin{enumerate}
\renewcommand{\labelenumi}{(\roman{enumi})}
\renewcommand{\theenumi}{\roman{enumi}}
\item\label{cateq} There is an equivalence of categories $\refl\La\simeq\refl\Ga$.
\item\label{gc} There is a generator-cogenerator $M\in\refl\La$ in the maximum exact structure such that $\Ga\simeq\End_\La(M)$.
\end{enumerate}
In this case, $\Ga$ also satisfies the two-sided $(2,2)$-condition.
\end{Thm}
\begin{proof}
	(\ref{cateq})$\Rightarrow$(\ref{gc})  Since $\Ga\in\refl\Ga$ is certainly a generator-cogenerator, we may set $M\in\refl\La$ to be the object corresponding to $\Ga\in\refl\Ga$.
	
	(\ref{gc})$\Rightarrow$(\ref{cateq})  This is a consequence of Theorem \ref{Morita}.
	
	The last assertion follows from Theorem \ref{max} as it shows the two-sided $(2,2)$-condition is preserved under reflexive equivalences.
\end{proof}

Since one can freely add summands for generator-cogenerators, we have the following consequence.
\begin{Ex}
Let $\La$ be a Noetherian rings satisfying the two-sided $(2,2)$-condition. Then for any $M\in\refl\La$ the module $\La\oplus M\in\refl\La$ is also a generator-cogenerator. It follows that whenever $\Ga=\End_\La(\La\oplus M)$ is Noetherian, we have a reflexive equivalence
\[ \xymatrix{ \refl\La\ar[r]^-\simeq& \refl\Ga }. \]
In particular, $\Ga$ also satisfies the two-sided $(2,2)$-condition.
\end{Ex}

Let us also discuss how to explain \cite[Proposition 2.4]{IR} in our context.
Let $R$ be a commutative Noetherian normal domain, and let $\La$ be a symmetric $R$-algebra (that is, $\La\simeq\Hom_R(\La,R)$ in $\Mod\La\otimes_R\La^\op$). We say $M\in\refl\La$ is a {\it height $1$ progenerator} if $M_\p\in\mod\La_\p$ is a progenerator for all $\p\in\Spec R$ of height $1$. %\comment{we do not need that $\La$ is symmetric, but we need that $\La$ is $(2,2)$.}
\begin{Prop}[{\cite[Proposition 2.4]{IR}}]
Let $M\in\refl\La$ be a height $1$ progenerator and $\Ga=\End_\La(M)$. Then the functor $\Hom_\La(M,-)\colon\refl\La\to\refl\Ga$ is an equivalence.
\end{Prop}
In view of Theorem \ref{Morita}, this is a consequence of the following observation.
\begin{Prop}
Any height $1$ progenerator is a generator-cogenerator.
\end{Prop}
\begin{proof}
	Suppose that $M\in\refl\La$ is height $1$ progenerator. Let $X\in\refl\La$ and take a right $(\add M)$-approximation $f\colon N\to X$. Since approximations are preserved under localizations, the map $N_\p\to X_\p$ is a right $(\add M_\p)$-approximation, and hence is surjective since $M_\p\in\proj\La_\p$. It follows that $\Coker f\in\mod\La$ has codimension $\geq2$, and thus $\sgrade\geq2$ since $R$ is normal. Therefore, $f$ is an (admissible) epimorphism in $\refl\La$ by Proposition \ref{confl}, and hence $M\in\refl\La$ is a generator.
	Now, if $M$ is a height $1$ progenerator, then so is $M^\ast\in\mod\La^\op$ by $(M^\ast)_\p=(M_\p)^\ast$. It follows from the previous step that $M^\ast\in\refl\La^\op$ is a generator, and then by duality we deduce that $M\in\refl\La$ is a cogenerator.
	%Suppose conversely that $M\in\refl\La$ is a generator-cogenerator. Then we have an epimorphism $M\to\La$ in $\refl\La$, in other words, an exact sequence $M\to\La\to C\to 0$ in $\mod\La$ with $\sgrade C\geq2$ by \ref{confl}. Since $C$ has codimension $\geq2$, we deduce that $M$ is a height $1$ generator.
\end{proof}
Note that the converse of the above implication does not hold. Indeed, let $R$ be an arbitrary commutative Noetherian domain and $\La=R[x]/(x^2)$. Then $R\oplus\La\in\refl\La$ is a generator-cogenerator but not a height $1$ progenerator (not even projective when localized at height $1$ primes).

\subsection{Proof}
The rest of this section is devoted to the proof of Theorem \ref{Morita}.
Let us start with some preparatory lemmas. Recall that a ring $A$ is {\it right coherent} (resp. {\it left coherent}) if the category $\proj A$ has weak kernels (resp. weak cokernels). A {\it coherent ring} is a ring which is left and right coherent. For a right coherent ring $A$, we denote by $\mod A$ the category of finitely presented right $A$-modules.
We prepare several lemmas which are variants or applications of standard methods for module categories, see e.g.\,\cite[Chapter II.2]{ARS}.
\begin{Lem}\label{funf}
Let $\E$ be an exact category, $M\in\E$, and $\La=\End_\E(M)$. %Suppose that $\La=\End_\La(M)$ is Noetherian.
\begin{enumerate}
\item\label{contra} Suppose that $M$ is a cogenerator and $\La$ is right coherent. Then for each $X\in\E$, we have $\Hom_\E(M,X)\in\mod\La$. In particular, $\add M$ is contravariantly finite.
\item Suppose that $M$ is a generator and $\La$ is left coherent. Then for each $X\in\E$, we have $\Hom_\E(X,M)\in\mod\La^\op$. In particular, $\add M$ is covariantly finite.
\end{enumerate}
\end{Lem}
\begin{proof}
	We only prove (1). Let $X\in\E$. Since $M$ is a cogenerator, we have a conflation $X\to M^0\to Y$ and an admissible monomorphism $Y\to M^1$ with $M^0, M^1\in\add M$. Applying $\Hom_\C(M,-)$ we get an exact sequence $0\to \Hom_\E(M,X)\to\Hom_\E(M,M^0)\to\Hom_\E(M,M^1)$ in $\Mod\La$. Since $\La=\End_\E(M)$ is right coherent, we see that $\Hom_\C(M,X)$ is finitely presented as a right $\La$-module. %and hence $X$ has a right $(\add M)$-approximation.
\end{proof}

\begin{Lem}\label{zensha}
Let $\E$ be an exact category and $M\in\E$ a generator-cogenerator such that $\La=\End_\E(M)$ is coherent.
\begin{enumerate}
\item For each $X\in\E$ there is an admissible epimorphism $M_0\to X$ which is a right $(\add M)$-approximation.
\item For each $X\in\E$ there is an admissible monomorphism $X\to M^0$ which is a left $(\add M)$-approximation.
\end{enumerate}
\end{Lem}
\begin{proof}
	We only prove (1). Since $M$ is a generator we have an admissible epimorphism $M_1\to X$, say, with kernel $Y$. Also, since $M$ is a cogenerator we have a right $(\add M)$-approximation $M_2\to X$ by Lemma \ref{funf}(\ref{contra}). 	Taking the pull-back of the admissible epimorphism $M_1\to X$ along this approximation $M_2\to X$, we get a commutative diagram below.
	\[ \xymatrix{
		0\ar[r]&Y\ar[r]\ar@{=}[d]& N\ar[r]\ar[d]&M_2\ar[d]\ar[r]&0\\
		0\ar[r]&Y\ar[r]&M_1\ar[r]&X\ar[r]& 0} \]
	By \cite[2.12]{Bu}, collapsing the right square gives a conflation, thus the map $M_1\oplus M_2\to X$ is an admissible epimorphism. This is also right $(\add M)$-approximation since it contains $M_2$ as a summand.
\end{proof}

The following step is crucial.
\begin{Lem}\label{refv}
Let $M\in\E$ be a generator-cogenerator such that $\La=\End_\E(M)$ is coherent. Then for every $X\in\E$ we have the following.
\begin{enumerate}
\item $\Hom_\E(M,X)^\ast\simeq\Hom_\E(X,M)$ in $\mod\La^\op$.
\item $\Hom_\E(X,M)^\ast\simeq\Hom_\E(M,X)$ in $\mod\La$.
\end{enumerate}
Therefore, there is a commutative diagram
\[ \xymatrix@R=3mm@C=20mm{
	&\mod\La\ar@{<->}[dd]^-{(-)^\ast}\\
	\E\ar[ur]^-{\Hom_\E(M,-)}\ar[dr]_-{\Hom_\E(-,M)}&\\
	&\mod\La^\op.} \]
%\begin{enumerate}
%\item If $M\in\C$ is a generator, then $\Hom_\C(M,X)^\ast\simeq\Hom_\C(X,M)$ in $\mod\La^\op$.
%\item If $M\in\C$ is a cogenerator, then $\Hom_\C(X,M)^\ast\simeq\Hom_\C(M,X)$ in $\mod\La$.
%\end{enumerate}
\end{Lem}
\begin{proof}
	We only prove (1). Since $M$ is a generator we have an admissible epimorphism $M_0\to X$, which we may assume to be a right $(\add M)$-approximation by Lemma \ref{zensha} since $M$ is a cogenerator as well. Taking the similar morphism for the kernel, we get a presentation
	\[ \xymatrix{ M_1\ar[r]& M_0\ar[r]&X\ar[r]& 0} \]
	which induces an exact sequence $\Hom_\E(M,M_1)\to\Hom_\E(M,M_0)\to\Hom_\E(M,X)\to0$ in $\mod\La$. Applying $(-)^\ast=\Hom_\La(-,\La)$ to this and comparing it with the one obtained by applying $\Hom_\E(-,M)$ to the above presentation, we get a commutative diagram
	\[ \xymatrix{
		0\ar[r]&\Hom_\E(M,X)^\ast\ar[r]\ar[d]&\Hom_\E(M,M_0)^\ast\ar[r]\ar[d]^-\rsimeq&\Hom_\E(M,M_1)\ar[d]^-\rsimeq\\
		0\ar[r]&\Hom_\E(X,M)\ar[r]&\Hom_\E(M_0,M)\ar[r]&\Hom_\E(M_1,M) } \]
	in $\mod\La^\op$, in which right two vertical maps are isomorphisms. We conclude that $\Hom_\E(M,X)^\ast\simeq\Hom_\E(X,M)$, as desired.
\end{proof}
	
Now we are ready to prove the main result Theorem \ref{Morita} of this section.
\begin{proof}[Proof of Theorem \ref{Morita}]
	Let $\C$ be a quasi-abelian category and $M\in\C$ be a generator-cogenerator such that $\End_\C(M)=\La$ is two-sided Noetherian.
	
	We first show that the functor $F=\Hom_\C(M,-)\colon\C\to\Mod\La$ is fully faithful. %Although this is standard, we include the proof for the convenience of the reader. 
	This is the standard ``generator correspondence''. 
	We have to prove that
	\[ \xymatrix{ F_{X,Y}\colon\Hom_\C(X,Y)\ar[r]& \Hom_\La(FX,FY) } \]
	is an isomorphism for every $X,Y\in\C$. This is clearly the case for $X\in\add M$. By Lemma \ref{zensha}, we may take an admissible epimorphism $M_0\to X$ which is a right $(\add M)$-approximation. Taking the similar morphism for the kernel, we get a presentation $M_1\to M_0\to X\to 0$ which yields an exact sequence $FM_1\to FM_0\to FM_0\to 0$. Applying $\Hom_\La(-,FY)$ and comparing with the one obtained by applying $\Hom_\C(-,Y)$ to the presentation, we get a commutative diagram
	\[ \xymatrix{
		0\ar[r]&\Hom_\C(X,Y)\ar[r]\ar[d]_-{F_{X,Y}}&\Hom_\C(M_0,Y)\ar[r]\ar[d]_-{F_{M_0,Y}}^-\rsimeq&\Hom_\C(M_1,Y)\ar[d]_-{F_{M_1,Y}}^-\rsimeq\\
		0\ar[r]&\Hom_\La(FX,FY)\ar[r]&\Hom_\La(FM_0,FY)\ar[r]&\Hom_\La(FM_1,FY), } \]
	in which the right two vertical maps are isomorphisms. We conclude by the five lemma that so is the remaining one, i.e.\! $F_{X,Y}$. This proves that $F$ is fully faithful.
	
	We know by Lemma \ref{funf} that the functor $F$ takes value in $\mod\La$, and moreover in $\refl\La$ by Lemma \ref{refv}. It remains to show that $\Hom_\C(M,-)\colon\C\to\refl\La$ is essentially surjective. Let $L\in\refl\La$ and pick an exact sequence $0\to L\to P^0\xrightarrow{f} P^1$ with $P^0, P^1\in\proj\La$. Let $M^0\xrightarrow{g} M^1$ be the morphism in $\add M\subset \C$ corresponding to $f$ under the equivalence $\Hom_\C(M,-)\colon\add M\xsimeq\proj\La$. Since $\C$ has kernels, we may set $X=\Ker g$ in $\C$, which then shows $L=\Hom_\C(M,X)$.
	
	We have proved that there is an equivalence $\C\simeq\refl\La$. We conclude by Theorem \ref{max} that $\La$ satisfies the two-sided $(2,2)$-condition.
\end{proof}
 
\begin{Rem}
We made no essential use of the condition $\C$ is quasi-abelian in the course of the above proof. In fact, for an arbitrary exact category $\E$ and a generator-cogenerator $M\in\E$ with coherent endomorphism ring $\La$, we have a fully faithful functor
\[ \xymatrix{ \Hom_\E(M,-)\colon\E\ar[r]&\refl\La } \]
which can be seen to be an equivalence, for example, when $\E$ has kernels.
However, this is not an optimal result, for example in the following points.
\begin{itemize}
\item We would like the equivalence to preserve the exact structures, but there is no known natural exact structure on $\refl\La$ for general $\La$.
\item Since the category $\refl\La$ does not necessarily have kernels (see Lemma \ref{EX-ker}), it is not natural to assume that $\E$ has kernels.
\end{itemize}
It would therefore be an interesting task to formulate a categorical structure (possibly weaker than an exact structure or perhaps just additive categories) and characterize the category of reflexive modules among them. 
\end{Rem}

\end{document}

%% file: definitions.tex
\DeclareMathOperator{\image}{Im}
\renewcommand{\Im}{\image}
\DeclareMathOperator{\Ker}{Ker}
\DeclareMathOperator{\Coker}{Coker}
\DeclareMathOperator{\coker}{coker}

\DeclareMathOperator{\Hom}{Hom}

\DeclareMathOperator{\End}{End}

\DeclareMathOperator{\Tor}{Tor}
\DeclareMathOperator{\Ext}{Ext}

\DeclareMathOperator{\pd}{proj.dim}

\DeclareMathOperator{\id}{inj.dim}

\DeclareMathOperator{\fld}{flat.dim}

\DeclareMathOperator{\gd}{gl.dim}
\DeclareMathOperator{\gldim}{gl.dim}

\DeclareMathOperator{\ddim}{dom.dim}
\DeclareMathOperator{\domdim}{dom.dim}

\DeclareMathOperator{\md}{mod}
\renewcommand{\mod}{\md}
\DeclareMathOperator{\Mod}{Mod}
\DeclareMathOperator{\smod}{\underline{\md}}

\DeclareMathOperator{\proj}{proj}

\DeclareMathOperator{\fl}{fl}
\DeclareMathOperator{\refl}{ref}

\DeclareMathOperator{\add}{add}

\DeclareMathOperator{\Tr}{Tr}

\DeclareMathOperator{\Ass}{Ass}

\DeclareMathOperator{\Spec}{Spec}

\DeclareMathOperator{\height}{ht}
\DeclareMathOperator{\depth}{depth}
\DeclareMathOperator{\grade}{grade}
\DeclareMathOperator{\sgrade}{s.\!grade}

\DeclareMathOperator{\GL}{GL}

\newcommand\recollement[3]{\xymatrix{{#1}\ar[r]&{#2}\ar[r]\ar@/_7pt/[l]\ar@/^7pt/[l]&{#3}\ar@/_7pt/[l]\ar@/^7pt/[l] }}
\newcommand\recollementwithmaps[9]{\xymatrix{{#1}\ar[r]|-{#8}&{#2}\ar[r]|-{#5}\ar@/_7pt/[l]_-{#7}\ar@/^7pt/[l]^-{#9}&{#3}\ar@/_7pt/[l]_-{#4}\ar@/^7pt/[l]^-{#6} }}
\def\A{\mathcal{A}}

\def\C{\mathcal{C}}
\def\D{\mathcal{D}}
\def\E{\mathcal{E}}

\renewcommand\S{\mathcal{S}}

\def\Ga{\Gamma}

\def\La{\Lambda}

\def\Si{\Sigma}
\def\Om{\Omega}

\def\p{\mathfrak{p}}

\def\op{\mathrm{op}}

\def\pprime{{\prime\prime}}
\def\aast{{\ast\ast}}

\def\rsimeq{\rotatebox{-90}{$\simeq$}}
\def\xsimeq{\xrightarrow{\simeq}}
\def\ysimeq{\xleftarrow{\simeq}}

\def\romanenum{\renewcommand\labelenumi{(\roman{enumi})}\renewcommand\theenumi{\roman{enumi}}}

\newtheorem{Thm}{Theorem}[section]
\newtheorem{Lem}[Thm]{Lemma}
\newtheorem{Prop}[Thm]{Proposition}
\newtheorem{Cor}[Thm]{Corollary}

\newtheorem{Prop-Def}[Thm]{Proposition-Definition}
\newtheorem{Thm-Def}[Thm]{Theorem-Definition}

\theoremstyle{definition}
\newtheorem{Def}[Thm]{Definition}
\newtheorem{Ex}[Thm]{Example}

\newtheorem{Qs}[Thm]{Question}

\theoremstyle{remark}
\newtheorem{Rem}[Thm]{Remark}

\newcounter{step}

\makeatletter

\@addtoreset{equation}{section}
\makeatother